\documentclass[dvips,preprint]{imsart}

\usepackage{amsthm,amsmath}

\usepackage{amssymb,amsmath,color,verbatim,bbm}
\usepackage{verbatim}
\usepackage{stmaryrd}
\usepackage{natbib}
\usepackage{graphicx}
\usepackage[numbered]{algo}
\usepackage{ifthen}

\numberwithin{equation}{section}

\newcommand{\adj}[2]{\psi_{#1}^{#2}}
\newcommand{\adjfunc}[1]{\Psi_{#1}}
\newcommand{\alg}{\mathcal{B}}

\newcommand{\ba}{\bar{\alpha}}
\newcommand{\bb}[1]{\mathbb{W}^{(#1)}}
\newcommand{\condexp}{;}

\newcommand{\ConvP}{\stackrel{\mathbb{P}}{\longrightarrow}}
\newcommand{\cset}{\mathsf{C}}
\newcommand{\csetnew}{\tilde{\mathsf{C}}}
\newcommand{\dd}{\mathrm d}
\newcommand{\ee}{\mathbb E}
\newcommand{\eg}{e.g.}
\newcommand{\eqsp}{\;}
\newcommand{\define}{\ensuremath{\stackrel{\mathrm{def}}{=}}}
\newcommand{\dtheta}{d_\theta}
\newcommand{\dX}{d_X}

\newcommand{\GPEAPS}{GPEPS}
\newcommand{\hd}{q}
\newcommand{\hk}{Q}
\newcommand{\ie}{i.e.}
\newcommand{\IMQ}{\mathcal Q }
\newcommand{\idx}[2]{I_{#1}^{#2}}
\newcommand{\init}{\chi}
\newcommand{\inverse}[1]{{#1}^{\leftarrow}}
\newcommand{\instr}[1]{\pi_{#1}^N}
\newcommand{\lemkern}{S}
\newcommand{\lk}{g}
\newcommand{\lkhd}[1]{\mathrm{L}_{#1}}
\newcommand{\llkhd}[1]{\ell_{#1}}

\newcommand{\logqdraw}[3]{\bar{V}_{#3}^{#2}(#1)}
\newcommand{\Lower}[1]{\tilde{W}_{#1}^-}
\newcommand{\Lp}[1]{\mathsf{L}^{#1}}
\newcommand{\lqkern}{\bar{P}}

\newcommand{\modmeas}[1]{\sm{#1}\langle \adjfunc{k} \rangle}
\newcommand{\myleft}{}
\newcommand{\myright}{}
\newcommand{\N}{\mathbb{N}}
\newcommand{\nbr}[2]{M_{#1}^{#2}}
\newcommand{\normpdf}{\mathcal{N}}

\newcommand{\overpil}[1]
{\raisebox{1.5ex}{ $\underrightarrow{\text{\ \scriptsize #1\ }}$ }}
\newcommand{\parti}[2]{\ifthenelse{\equal{#1}{}}{\xi_{N,#2}}{\xi_{#1}^{#2}}}
\newcommand{\partinew}[2]{\ifthenelse{\equal{#1}{}}{\tilde{\xi}_{N,#2}}{\tilde{\xi}_{#1}^{#2}}}
\newcommand{\partmixture}[1]{\bar{\phi}^N_{#1}}
\newcommand{\path}{\tilde{W}}
\newcommand{\partalg}{\mathcal{F}}
\newcommand{\partalgp}{\mathcal{F}}
\newcommand{\partalgtd}{\bar{\mathcal{F}}}
\newcommand{\partdraw}[1]{X_{#1}}

\newcommand{\partsm}[1]{\phi_{#1}^N}
\newcommand{\partspace}{\boldsymbol{\Xi}}
\newcommand{\partspacenew}{\tilde{\boldsymbol{\Xi}}}
\newcommand{\pp}{\mathbb P}
\newcommand{\proposal}[1]{R_{#1}}
\newcommand{\propdens}[1]{r_{#1}}
\newcommand{\qdraw}[3]{V_{#3}^{#2}(#1)}
\newcommand{\qdrawnew}[3]{V_{#3}^{#2}(#1)}

\newcommand{\qest}[2]{q^{#1}_{#2}}
\newcommand{\qkern}{P}
\newcommand{\revM}[1]{\overleftarrow{\hk}_{#1}}
\newcommand{\rr}{\mathbb R}
\newcommand{\selpart}[2]{\bar{\xi}_{#1}^{#2}}
\newcommand{\sm}[1]{\phi_{#1}}

\renewcommand{\ss}[1]{\mathbb{S}^{(#1)}}

\newcommand{\Unif}{\mathrm{Unif}}
\newcommand{\U}[1]{U_{N,#1}}
\newcommand{\Up}[1]{U_{N,#1}}
\newcommand{\Upper}[1]{\tilde{W}_{#1}^+}
\newcommand{\Utd}[1]{\bar{U}_{N,#1}}
\newcommand{\ud}{\mathrm{d}}
\newcommand{\uk}[1]{L_{#1}}
\newcommand{\ukvar}{L}
\newcommand{\Xnew}[1]{\tilde{X}_{#1}}
\newcommand{\Xalg}{\mathcal{X}}
\newcommand{\Xset}{\mathsf{X}}
\newcommand{\Yalg}{\mathcal{Y}}
\newcommand{\Yset}{\mathsf{Y}}
\newcommand{\wgt}[2]{\ifthenelse{\equal{#1}{}}{\omega_{N,#2}}{\omega_{#1}^{#2}}}
\newcommand{\wgtfunc}[1]{\Phi_{#1}}
\newcommand{\wgtfuncest}[2]{\Phi_{#1}^{#2}}
\newcommand{\wgtnew}[2]{\ifthenelse{\equal{#1}{}}{\tilde{\omega}_{N,#2}}{\tilde{\omega}_{#1}^{#2}}}
\newcommand{\wgtsum}[1]{\ifthenelse{\equal{#1}{}}{\Omega_N}{\Omega_{#1}^N}}

\newcommand{\ww}[1]{\mathbb{W}^{(#1)}}

\newcounter{hyp}
\newenvironment{hyp}[1]{\refstepcounter{hyp}\it\begin{itemize}\item[{\it
      (A\arabic{hyp})}] \label{hyp:#1}}{\end{itemize}}
\newcommand{\refhyp}[1]{{\it (A\ref{hyp:#1})}}
\newcommand{\refhyps}[2]{{\it (A\ref{hyp:#1}--\ref{hyp:#2})}}

\newtheorem{definition}{Definition}[section]
\newtheorem{lemma}{Lemma}[section]
\newtheorem{proposition}{Proposition}[section]
\newtheorem{theorem}{Theorem}[section]

\endlocaldefs

\begin{document}

\begin{frontmatter}

    \title{Particle-based likelihood inference in partially observed diffusion processes using generalised Poisson estimators}%\protect\thanksref{T1}}
    \runtitle{Particle-based inference in partially observed diffusions}
    %\thankstext{T1}{Footnote to the title with the `thankstext' command.}

    \begin{aug}

        \author{\fnms{Jimmy} \snm{Olsson}%\thanksref{t1,t2}
        \ead[label=e1]{jimmy@maths.lth.se}}
        \and
        \author{\fnms{Jonas} \snm{Str\"ojby}\corref{t1}%\thanksref{t3}
        \ead[label=e2]{strojby@maths.lth.se}}

        \address{Center of Mathematical Sciences\\Lund University\\Lund, Sweden\\
        \printead{e1,e2}}

        \runauthor{J.~Olsson and J.~Str\"ojby}

    \end{aug}

    \begin{abstract}
        This paper concerns the use of the expectation-maximisation (EM) algorithm for inference in partially observed diffusion processes. In this context, a well known problem is that all except a few diffusion processes lack closed-form expressions of the transition densities. Thus, in order to estimate efficiently the EM intermediate quantity we construct, using novel techniques for \emph{unbiased} estimation of diffusion transition densities, a random weight fixed-lag auxiliary particle smoother, which avoids the well known problem of particle trajectory degeneracy in the smoothing mode. The estimator is justified theoretically and demonstrated on a simulated example.
    \end{abstract}
    


    %history:
    %\received{\smonth{1} \syear{0000}}
    \tableofcontents

\end{frontmatter}

\section{Introduction}
\label{section:introduction}
In this paper we discuss the use of \emph{sequential Monte Carlo} (SMC) methods (alternatively termed \emph{particle methods}) for likelihood-based inference in \emph{partially observed diffusions} (PODs). The proposed method relies on a novel approach for estimating transition densities of diffusion processes via so-called \emph{generalised poisson estimators} (GPEs). For the models under consideration, the likelihood function of the observed data cannot be expressed on closed-form; however, since partially observed diffusion models are, like more general latent variable models, specified using conditional dependence relations, this inference problem can be efficiently cast into the framework of the \emph{expectation-maximisation} (EM) \emph{algorithm} proposed by \cite{dempster:laird:rubin:1977}.
When applying the EM algorithm in the POD context there are two main difficulties: firstly, in all except a few cases, the transition density of the diffusion process, and thus the complete data log-likelihood function, lacks an analytic expression; secondly, computing the \emph{intermediate quantity} of the expectation-step involves taking expectations under the \emph{smoothing distribution}, \ie\ the conditional distribution of the hidden states at the observation time points given the observed data record, which is not---even in the case of a known transition density---available on closed-form. These two issues make, as documented by several authors, MLE-based inference in PODs very challenging. In this paper we address these problems by applying the GPE suggested \citep[as a refinement of results obtained in][]{beskos:papaspiliopoulos:roberts:fearnhead:2006} by \cite{fearnhead:papaspiliopoulos:roberts:2008} in conjunction with SMC smoothing algorithms. Unfortunately, it has been observed by several authors that using standard SMC methods in the smoothing mode may be unreliable for larger observation sample sizes $n$, since resampling systematically the particles leads to degeneracy of the particle paths. As a solution, we adapt the \emph{fixed-lag smoother} proposed by \cite{olsson:cappe:douc:moulines:2008} to the framework of PODs. This technique relies, in the spirit of \cite{kitagawa:1998}, on \emph{forgetting properties} of the conditional hidden chain; by this is meant that the hidden chain forgets its past when evolving, backwards as well as forwards, conditionally on the given observation sequence. The constructed algorithm avoids efficiently particle trajectory degeneracy at the cost of a bias which can however be controlled by a suitable choice of the introduced lag parameter.

In order to obtain a high performance of the particle smoother it is in general necessary to propose (mutate) the particles according a kernel that takes the information provided by the current observation into account; indeed, mutating, as in the \emph{bootstrap particle filter}, the particles ``blindly'' according to the dynamics of the hidden Markov chain will often lead to severe degeneracy of the particle importance weights. However, such an improved proposal strategy is not straightforwardly adopted to PODs, since computing the resulting importance weights involves computing a ratio of the transition density of the hidden diffusion process (for which a closed-form expression is missing in general) and that of the chosen proposal kernel. To cope with this, we follow \cite{fearnhead:papaspiliopoulos:roberts:2008} and replace each evaluation of the hidden process transition density by a draw from the GPE. Thus, the GPE serves two purposes in our algorithm as it is used, firstly, for computing unbiased estimates of particle importance weights for a particle filter based on a proposal kernel different from the transition kernel of the hidden diffusion process and, secondly, for estimating the EM intermediate quantity itself.

The contribution of our study is fourfold, since the proposed intermediate quantity estimator
\begin{enumerate}
\item approximates efficiently the expectation step in a \emph{single sweep} of the data record, yielding an algorithm with a computational complexity of order $\mathcal{O}(n N)$;
\item copes, as it is not based on any Euler discretisation or linearisation technique, efficiently with model nonlinearities;
\item has only limited computer data storage requirements, which is essential in, \eg, high frequency applications where sometimes very long measurement sequences are considered;
\item is provided with a rigorous convergence result describing its convergence to the true intermediate quantity. This result is derived via a convergence result, obtained under minimal assumptions, for the GPE-based particle smoother.
\end{enumerate}

For models exhibiting poor mixing properties, in which case we cannot expect a high performance of the fixed-lag smoother, we propose an alternative algorithm where the GPE is used in conjunction with the particle-based \emph{forward-filtering backward-smoothing} procedure proposed by \cite{godsill:doucet:west:2004}. This scheme, which relies on a decomposition of the smoothing measure that incorporates the so-called \emph{backward kernels} (\ie\ the transition kernels of the hidden Markov chain when evolving backwards in time and conditionally on the observations) of the model, avoids particle path degeneracy completely through an additional simulation pass in the time-reversed direction. Moreover, it does not suffer from the additional, model dependent bias of the fixed-lag smoother. However, these appealing properties are obtained at the cost of a significant increase of computational work, since the complexity of the scheme in question is quadratic in the number of particles.

The paper is organised as follows: In Section~\ref{section:preliminaries} we recall the concept of PODs and discuss likelihood-based inference in such models via data augmentation and the EM-algorithm. GPEs are described in Section~\ref{section:GPEs} and Section~\ref{section:appendix:EA}, and Section~\ref{section:particle:smoothing} is devoted to SMC smoothing in general. In Sections~\ref{section:fixed-lag} and \ref{section:FFBS} we introduce the fixed-lag smoother and the forward-filtering backward-simulation smoother, respectively; moreover, we discuss how these techniques can be adjusted to PODs using GPEs. A theoretical result describing the convergence of the fixed-lag-based estimator is found in Section~\ref{section:theoretical:results}, and in Section~\ref{section:simulation:study} we illustrate the method on partially observed log-growth and genetics diffusion models. In Section~\ref{section:conclusion}, the paper is concluded by some final conclusions and remarks. Proofs are found in Section~\ref{section:proofs}.

\section{Preliminaries}
\label{section:preliminaries}
In the following we assume that all random variables are defined on a common probability space $(\Omega, \mathcal{F}, \pp)$ and let $\ee$ denote expectations associated with $\pp$. Denoting by $\mathbbm{1}$ the indicator function and letting $X$ be any random variable on $(\Omega, \mathcal{F})$, we will often make use of the short-hand notation $\ee{}[X ; A] = \ee{}[X \mathbbm{1}_A]$. Let $X \define (X_t)_{t \geq 0}$ be continuous-time diffusion process taking values in some space $(\Xset, \Xalg)$, with $\Xset \subseteq \rr^{\dX}$. More specifically, the dynamics of the process is governed by the the stochastic differential equation
\begin{equation} \label{Eq:SDE1}
    \dd X_t = \mu(X_t, \theta) \, \dd t + \sigma(X_t, \theta) \, \dd W_t \eqsp,
\end{equation}
where $W \define (W_t)_{t \geq 0}$ is Brownian motion. We denote by $\ww{x}$ the law of $W$ given that $W_0 = x$ and let $(\mathcal{F}_t)_{0 \leq t}$ be the filtration generated by $W$. The functions $\mu(\cdot, \theta)$ and $\sigma(\cdot, \theta)$ are assumed to satisfy regularity conditions (locally Lipschitz with a linear growth bound) that guarantee a weakly unique, global solution of \eqref{Eq:SDE1}. We will consider a framework where the process $X$ is only partially observed at discrete time points $(t_k)_{k \geq 0}$ through the process $Y \define (Y_k)_{k \geq 0}$ taking values in some measurable space  $(\Yset, \Yalg)$. The observations of $Y$ are assumed to be, conditionally on the latent process $X$, independent and such that the conditional distribution $G_\theta$ of $Y_k$ given $X$ depends on $X_{t_k}$ only. In the following we write, in order to simplify the notation, $X_k$ instead of $X_{t_k}$. The dynamics of the diffusion as well as the measurement process depend on some unknown model parameter $\theta$ which is assumed to belong to some compact parameter space $\Theta \subseteq \rr^{\dtheta}$.  Our main target is to estimate $\theta$ using the maximum likelihood method. For simplicity we assume that the observation time points are \emph{equally spaced} and denote by $Q_\theta$ and $\init$ the transition kernel and initial distribution, respectively, of the time homogeneous Markov chain $(X_k)_{k \geq 0}$. The family $( Q_\theta(x, \cdot) ; x \in \Xset, \theta \in \Theta )$ is dominated by the Lebesque-measure $\lambda$ with corresponding Radon-Nikodym derivatives $( q_\theta(x, \cdot) ; x \in \Xset, \theta \in \Theta )$. Moreover, suppose that $G_\theta$ has a density function $g_\theta$ with respect to some measure $\mu$ on $(\Yset, \Yalg)$ such that, for $k \geq 0$,
\begin{equation*}
    \pp(Y_k \in A | X_k) = \int_A g_\theta(X_k, y) \, \mu(\dd y) \eqsp, \quad  A \in \Yalg \eqsp.
\end{equation*}
Given a record $Y_{0:n} = (Y_0, Y_1, \ldots, Y_n)$ (similar vector notation will be used also for other quantites) of observations, a consistent estimate of the parameter $\theta$ is ideally formed by maximising the \emph{observed data likelihood function} $\llkhd{n}(\theta; Y_{0:n}) \define \log \lkhd{n}(\theta; Y_{0:n})$, where
\begin{equation*} %\label{eq:def:likelihood}
    \lkhd{n}(\theta; Y_{0:n}) \define \idotsint g_\theta(x_0, Y_0) \, \init(\dd x_0) \prod_{k = 1}^n g_\theta(x_k, Y_k) \, Q_\theta(x_{k - 1}, \dd x_k) \eqsp,
\end{equation*}
A problem with this approach is that we in general cannot compute $\lkhd{n}$ on closed-form, since this involves the evaluation of a high-dimensional integral over a complicated integrand. Since the partially observed diffusion model above is, like more general latent variable models, specified using conditional dependence relations, computation of parameter posterior distributions is facilitated significantly by maximising instead the \emph{complete data} log-likelihood function by means of the EM algorithm: Assume that we have at hand an initial estimate $\theta'$ of the parameter vector. In the EM algorithm an improved estimate is obtained by computing and maximising the intermediate quantity $\IMQ(\theta; \cdot)$ defined by
\begin{equation} \label{Eq:IMQ}
    \IMQ_n(\theta; \theta') \define \ee_{\theta'} \left[ \left. \sum_{k=0}^{n-1} \log q_\theta (X_k, X_{k+1}) \right| Y_{0:n} \right]  + \ee_{\theta'} \left[ \left. \sum_{k = 0}^n \log g_\theta (X_k, Y_k) \right| Y_{0:n} \right] \eqsp.
\end{equation}
Here we have written $\ee_{\theta'}$ to stress that the expectations are taken under the dynamics determined by the initial parameter $\theta'$. Under weak assumptions, repeating recursively this procedure yields a sequence of parameter estimates that converges to a stationary point $\theta^*$ of the observed data log-likelihood \citep{wu:1983}. As clear from \eqref{Eq:IMQ}, computing $\IMQ_n$ requires the computation of expected values under the smoothing distribution, \ie\ the distribution of the state sequence $X_{0:n}$ conditionally on the observations $Y_{0:n}$, given by, for $A \in \Xalg^{\varotimes(n + 1)}$,
\begin{equation} \label{eq:def:smoothing:dist}
    \sm{n}(A  ; \theta) \define  \frac{\idotsint_A
    \lk_\theta(x_0, Y_0) \, \init(\ud x_0) \prod_{k = 1}^n \lk_\theta(x_k, Y_k) \, \hk_\theta(x_{k - 1}, \ud x_k)}{\lkhd{n}(\theta ; Y_{0:n})} \eqsp.
\end{equation}
Of special interest is the \emph{filter distribution}, \ie\ the distribution of $X_n$ conditionally on $Y_{0:n}$, given by the restriction $\sm{n|n}(A) \define \sm{n}(\Xset^n \times A)$, $A \in \Xalg$, of the smoothing distribution to the last component.
It is easily shown that the flow $(\sm{k})_{k = 0}^\infty$ satisfies the well-known \emph{forward smoothing recursion}
\begin{equation} \label{eq:smoothing:recursion}
    \sm{k + 1}(A ; \theta) =  \frac{\lkhd{k}(\theta ; Y_{0:k})}{\lkhd{k + 1}(\theta ; Y_{0:k + 1})} \iint_A \lk_\theta(x_{k + 1}, Y_{k + 1}) \, \hk_\theta(x_k, \ud x_{k + 1}) \, \sm{k}(\ud x_{0:k} ; \theta) \eqsp,
\end{equation}
where $A \in \Xalg^{\otimes (k + 2)}$. By introducing the (non-Markovian) transition kernel
\[
    \uk{k}(x_k, A ; \theta) \define \int_A \lk_\theta(x_{k + 1}, Y_{k + 1}) \, \hk_\theta(x_k, \ud x_{k + 1}) \eqsp,
\]
for $x_k \in \Xset$ and $A \in \Xalg$, we may rewrite the recursion \eqref{eq:smoothing:recursion} as
\begin{equation} \label{eq:smoothing:recursion:alt}
    \sm{k + 1}(A ; \theta) = \frac{\iint_A \uk{k}(x_k, \ud x_{k + 1} ; \theta) \, \sm{k}(\ud x_{0:k} ; \theta)}{\iint \uk{k}(x_k, \ud x_{k + 1} ; \theta) \, \sm{k}(\ud x_{0:k} ; \theta)} \eqsp.
\end{equation}
Here the normalised (Markovian) kernel $\uk{k}(x_k, A ; \theta) / \uk{k}(x, \Xset ; \theta)$ is the so-called \emph{optimal kernel} describing the distribution of $X_{k + 1}$ given $X_k = x_k$ \emph{and} the new observation $Y_{k + 1}$.

In general, a closed-form solution of the recursion \eqref{eq:smoothing:recursion} is not available. A standard approach is thus to apply some SMC smoothing algorithm (described in in Section~\ref{section:particle:smoothing}) to approximate the expectations in \eqref{Eq:IMQ}. Unfortunately, both the SMC smoother itself as well as the intermediate quantity \eqref{Eq:IMQ} call for the transition density $q_\theta$, which is usually unknown except in a few special cases. Nevertheless, results obtained by \citet{beskos:papaspiliopoulos:roberts:fearnhead:2006} and \citet{fearnhead:papaspiliopoulos:roberts:2008} offer a method for estimating this density \emph{without bias}. A full treatment of this technique---which is a key ingredient of the estimation technique proposed here---is beyond the scope of this paper; nevertheless, the main framework and assumptions are described briefly in the next section. In addition, some more details can be found in Appendix~\ref{section:appendix:EA}.

\subsection{Generalised Poisson estimators}
\label{section:GPEs}
Define the function
\begin{equation*}
    \eta(\cdot, \theta) : u \mapsto \int^u \frac{1}{\sigma(v, \theta)} \, \dd v \eqsp,
\end{equation*}
and set $\Xnew{t} \define \eta(X_t, \theta)$. Denote by $\inverse{f}$ the inverse of any invertable function $f$. By applying It$\hat{\mathrm{o}}$'s formula we obtain the stochastic differential equation
\begin{equation} \label{Eq:SDE2}
    \dd \Xnew{t} = \alpha(\Xnew{t},\theta) \, \dd t + \dd W_t \eqsp,
\end{equation}
where
\begin{equation*} \label{eq:def:alpha}
    \alpha(u, \theta) \define \frac{\mu \{ \inverse{\eta}(u, \theta), \theta \}}{\sigma \{ \inverse{\eta}(u, \theta), \theta \}} + \frac{1}{2} \sigma' \{ \inverse{\eta}(u, \theta), \theta \} \eqsp,
\end{equation*}
for the transformed process $\Xnew{} \define (\Xnew{t})_{t \geq 0}$. Using again the notation $\Xnew{k} = \Xnew{t_k}$, let $\tilde{q}_\theta$ be the transition density (with respect to the Lebesgue measure $\lambda$) of $(\Xnew{k})_{k \geq 0}$. Then, straightforwardly,
\begin{equation} \label{eq:trans:dens:id}
    q_\theta(x, x') = \tilde{q}_\theta(x, x') |\eta'(x', \theta)| \eqsp.
\end{equation}
Assume the following:
\begin{hyp}{Beskos:C0}
The process $(M_t)_{t \geq 0}$, with
\begin{equation*} \label{Eq:RN_ds/dw}
M_t \define \exp \left(\int_0^t \alpha(\Xnew{s} , \theta) \, \dd \Xnew{s} + \int_0^t \alpha^2(\Xnew{s} , \theta ) \, \dd s \right) \eqsp,
\end{equation*}
is a martingale with respect to $\ww{x}$;
\end{hyp}
\begin{hyp}{Beskos:C1}
    $\alpha(\cdot, \theta)$ is continuously differentiable;
\end{hyp}
\begin{hyp}{Beskos:C2}
    $\alpha^2(\cdot, \theta) + \alpha'(\cdot, \theta)$ is bounded from below by some function $l(\theta)$.
\end{hyp}

Under these conditions, the GPE approach developed by \cite{fearnhead:papaspiliopoulos:roberts:2008} makes it possible to generate random variables $\tilde{V}_\theta(x, x')$ with $\ee \tilde{V}_\theta(x, x') = \tilde{q}_\theta(x, x')$ for any $(x, x') \in \Xset^2$, \ie\  $\tilde{V}_\theta(x, x')$ estimates the transition density $\tilde{q}_\theta$ without any bias, for a large class of diffusions of type \eqref{Eq:SDE2}. Then, letting $\qdraw{x, x'}{}{\theta} \define \tilde{V}_\theta(x, x') |\eta'(x', \theta)|$ yields, using \eqref{eq:trans:dens:id}, $\ee \qdraw{x, x'}{}{\theta} = q_\theta(x, x')$.
A full description of GPEs is beyond the scope of this paper; however, its main features are discussed in Appendix~\ref{section:appendix:EA}. In this paper we represent the GPE by a kernel $\qkern_\theta$ in sense that $\qdraw{x, x'}{}{\theta} \sim \qkern_\theta(x, x', \cdot)$. Similarly, using the related \emph{exact algorithm} developed by \cite{beskos:papaspiliopoulos:roberts:fearnhead:2006}, it is possible to construct a kernel $\lqkern_\theta$ such that $\ee \logqdraw{x, x', \theta}{}{\theta} = \log q_\theta(x, x') $ for draws $\logqdraw{x, x', \theta}{}{\theta} \sim \lqkern_\theta(x, x', \cdot)$. Appealingly, it is in many cases (see Section~\ref{section:simulation:study} for examples) possible to construct $\qkern_\theta$ and $\lqkern_\theta$ such that the functions $\theta \mapsto \qdraw{x, x'}{}{\theta}(\omega)$ and $\theta \mapsto \logqdraw{x, x'}{}{\theta}(\omega)$ are continuous for any fixed outcome $\omega \in \Omega$, yielding unbiased estimates of $q_\theta$ and $\log q_\theta$ \emph{for all} $\theta \in \Theta$ \emph{simultaneously}. This useful property makes, as we will see, the GPE approach well suited to numerical (log-)likelihood function optimisation.

%%%%%%%%%%%%%%%%%%%%%%%%%%%%%%%%%%%%%%%%%%%%%%%%%%%%%%%%%%
%%%% Particle smoothing - might be shortened considerably?
%%%%%%%%%%%%%%%%%%%%%%%%%%%%%%%%%%%%%%%%%%%%%%%%%%%%%%%%%%

\subsection{GPE-based particle smoothing}
\label{section:particle:smoothing}

Since we in this part deal with the problem of sampling $\sm{k}(\cdot ; \theta)$ for a given \emph{fixed} parameter value, we will throughout this section expunge $\theta$ from the notation. To begin with, we assume that we know the transition kernel density $q$.

In order to describe precisely how SMC methods may be used for producing approximate solutions to the smoothing recursion \eqref{eq:smoothing:recursion}, we suppose that we are given a weighted sample $(\parti{0:k|k}{i}, \wgt{k}{i})_{i = 1}^N$ of particle and associated weights, each particle $\parti{0:k|k}{i} = (\parti{1|k}{i}, \ldots, \parti{k|k}{i})$ being a random variable in $\Xset^{k + 1}$, approximating $\sm{k}$ in the sense that
\begin{equation} \label{eq:targeting}
    \partsm{k}(f) \define \left( \wgtsum{k} \right)^{-1} \sum_{i = 1}^N \wgt{k}{i} f(\parti{0:k|k}{i}) \approx \sm{k}(f) \eqsp,
\end{equation}
where $\wgtsum{k} \define \sum_{\ell = 1}^N \wgt{k}{\ell}$, for a large class of estimand functions $f$ on $\Xset^{k + 1}$. Now, in order to form an updated particle sample approximating $\sm{k + 1}$, as a new observation $Y_{k + 1}$ becomes available, a natural approach is to replace $\sm{k}$ in \eqref{eq:smoothing:recursion:alt} by its particle approximation. This yields the mixture (recall the notation $\delta_a$ for a Dirac mass located at $a$)
\begin{equation*}
    \partmixture{k + 1}(A) \define \sum_{i = 1}^N \frac{\wgt{k}{i} \uk{k}(\parti{k|k}{i}, \Xset)}{\sum_{\ell = 1}^N \wgt{k}{\ell} \uk{k}(\parti{k|k}{\ell}, \Xset)} \int_A \frac{\uk{k}(\parti{k|k}{i}, \ud x_{k + 1})}{\uk{k}(\parti{k|k}{i}, \Xset)} \, \delta_{\parti{0:k|k}{i}}(\ud x_{0:k}) \eqsp,
\end{equation*}
for $A \in \Xalg^{\otimes (k + 2)}$. Now, the aim is to simulate a new set of particles from $\partmixture{k + 1}$ and repeat this recursively to obtain particle samples approximating the smoothing distributions at all time steps. However, since we in general cannot neither simulate draws from the optimal kernel nor compute the mixture weights $\uk{k}(\parti{k|k}{i}, \Xset)$, we apply importance sampling and draw new particles from the instrumental mixture distribution
\begin{equation*}
    \instr{k + 1}(A) \define \sum_{i = 1}^N \frac{\wgt{k}{i} \adj{k}{i}}{\sum_{\ell = 1}^N \wgt{k}{\ell} \adj{k}{\ell}} \int_A \delta_{\parti{0:k|k}{i}}(\ud x_{0:k}) \, \proposal{k}\left( \parti{k|k}{i}, \ud x_{k + 1} \right) \eqsp,
\end{equation*}
for $A \in \Xalg^{\otimes (k + 2)}$, where $\proposal{k}$ is a Markovian proposal kernel and $(\adj{k}{i})_{i = 1}^N$ are positive numbers referred to as \emph{adjustment multiplier weights}. We will from now on assume that $\adj{k}{i} = \adjfunc{k}(\parti{0:k|k}{i})$ for some nonnegative function $\adjfunc{k} : \Xset^{k + 1} \rightarrow \rr^+$ and that each kernel $\proposal{k}$ has a density $\propdens{k}$ with respect to $\lambda$. Simulating a particle $\parti{0:k + 1|k + 1}{i}$ from $\instr{k + 1}$ is easily done by, firstly, drawing, according to the probability distribution proportional to $(\wgt{k}{i} \adj{k}{i})_{i = 1}^N$, a mixture component (or ancestor) index $\idx{k}{i}$ among $\{1, \ldots, N\}$ and, secondly, extending the selected ancestor with a draw from the proposal kernel, \ie\ letting $\parti{0:k + 1|k + 1}{i} \define (\parti{0:k|k}{\idx{k}{i}}, \parti{k + 1|k + 1}{i})$ with
$\parti{k + 1|k + 1}{i} \sim \proposal{k}(\parti{k|k}{\idx{k}{i}}, \cdot)$. After this, the drawn particle is assigned the importance weight
\begin{equation} \label{eq:weighting:step}
    \wgt{k + 1}{i} \define \wgtfunc{k + 1}\left( \parti{0:k + 1|k + 1}{i} \right) \eqsp,
\end{equation}
where, for $x_{0:k + 1} \in \Xset^{k + 2}$,
$$
    \wgtfunc{k + 1}(x_{0:k + 1}) \define \lk(x_{k + 1}, Y_{k + 1}) \adjfunc{k}^{-1}(x_{0:k}) \frac{\hd(x_k, x_{k + 1})}{\propdens{k}(x_k, x_{k + 1})} \eqsp,
$$
implying $\wgt{k + 1}{i} \propto \ud \partmixture{k + 1} / \ud \instr{k + 1}(\parti{0:k + 1|k + 1}{i})$. Finally, the weighted particle sample formed by the updated particles and weights is returned as an approximation of $\sm{k + 1}$. Moreover, since the filter distribution is the marginal of the smoothing distribution with respect to the last component, an estimate of $\sm{k + 1|k + 1}$ is formed by the marginal sample $(\parti{k + 1|k + 1}{i}, \wgt{k + 1}{i})_{i = 1}^N$.

Proposing and selecting the particles according to the dynamics of the latent process, \ie\ without making use of the information about the current state provided by the current observation, by letting $\proposal{k} \equiv \hk$ and $\adjfunc{k} \equiv \mathbf{1}$ for all $k$, corresponds to the bootstrap particle filter proposed by \citet{gordon:salmond:smith:1993}.

The algorithm, which was developed gradually by, mainly, \citet{handschin:mayne:1969}, \citet{gordon:salmond:smith:1993}, and \citet{pitt:shephard:1999}, will be referred to as the \emph{auxiliary particle smoother} (APS). In the setting of a partially observed diffusion process we do not have access to a closed-form expression of the transition density $\hd$, which is needed when evaluating the importance weight function $\wgtfunc{k + 1}$. However, the GPE makes it possible to estimate this density without bias via the kernel $\qkern$. This yields following algorithm, in following referred to as the \emph{GPE-based particle smoother} (\GPEAPS), in which $\hd$ in the weighting operation \eqref{eq:weighting:step} is replaced by the Monte Carlo estimate
\begin{equation} \label{eq:def:qest}
    \qest{\alpha}{}(x, x') \define \frac{1}{\alpha} \sum_{\ell = 1}^\alpha \qdraw{x, x'}{\ell}{} \eqsp,
\end{equation}
where the $\qdraw{x, x'}{\ell}{}$'s are drawn independently from $\qkern(x, x', \cdot)$.
Denote by
\begin{equation} \label{eq:def:wgtfuncest}
    \wgtfuncest{k + 1}{\alpha}(x_{0:k + 1}) \define \lk(x_{k + 1}, Y_{k + 1}) \adjfunc{k}^{-1}(x_{0:k}) \frac{\qest{\alpha}{}(x_k, x_{k + 1})}{\propdens{k}(x_k, x_{k + 1})} \eqsp,
\end{equation}
the resulting estimated importance weight function. One iteration of the \GPEAPS\ is described in detail in the following scheme.

\bigskip

\begin{algorithm}{GPEAPS}{\label{alg:SMC:smoothing}
    \qcomment{One iteration of \GPEAPS}
    \qinput{$(\parti{0:k|k}{i}, \wgt{k}{i})_{i = 1}^N$,
    $\proposal{k}$, $\alpha$}}
    \qfor $i \qlet 1$ \qto $N$ \\
    simulate $\idx{k}{i} \sim (\wgt{k}{j} \adj{k}{j} / \sum_{\ell = 1}^N \wgt{k}{\ell} \adj{k}{\ell})_{j = 1}^N$; \\
    simulate $\parti{k + 1|k + 1}{i} \sim \proposal{k}(\parti{k|k}{\idx{k}{i}}, \cdot)$; \\
    set $\parti{0:k + 1|k + 1}{i} \qlet (\parti{0:k|k}{\idx{k}{i}}, \parti{k + 1|k + 1}{i})$; \\
    simulate $\qdraw{\parti{k:k + 1|k + 1}{i}}{1:\alpha}{} \sim \qkern^{\varotimes \alpha}(\parti{k:k + 1|k + 1}{i}, \cdot)$; \\
    compute $\wgtfuncest{k + 1}{\alpha}$ via \eqref{eq:def:wgtfuncest}; \\
    set $\wgt{k + 1}{i} \qlet \wgtfuncest{k + 1}{\alpha}(\parti{k:k + 1|k + 1}{i})$; \qrof \\
    \qreturn $(\parti{0:k + 1|k + 1}{i}, \wgt{k + 1}{i})_{i = 1}^N$.
\end{algorithm}

\bigskip

Here we have used the notations $\qdraw{x, x'}{1:\alpha}{} \define (\qdraw{x, x'}{1}{}, \ldots, \qdraw{x, x'}{\alpha}{})$ and $\qkern^{\varotimes \alpha}(x, x', \cdot) \define \qkern(x, x', \cdot) \varotimes \cdots \varotimes \qkern(x, x', \cdot)$ ($\alpha$ times). Algorithm~\ref{alg:SMC:smoothing} extends the \emph{random weight auxiliary particle filter} proposed by \citet{fearnhead:papaspiliopoulos:roberts:2008} to the smoothing mode. Note that we have, in the scheme above, suppressed the dependence of the particles and the particle weights on $\alpha$ from the notation for clarity.

In the selection operation of Step~(2), each particle index is drawn from the probability distribution formed by the adjusted weights $(\wgt{k}{j} \adj{k}{j} / \sum_{\ell = 1}^N \wgt{k}{\ell} \adj{k}{\ell})_{j = 1}^N$. Letting $\nbr{k}{i}$ denote the number of times that index $i$ was drawn, the selection operation may be alternatively expressed as
\begin{equation} \label{eq:mult:selection}
    (\nbr{k}{1}, \ldots, \nbr{k}{N}) \sim \operatorname{Mult}\left( N, \left( \frac{\wgt{k}{j} \adj{k}{j}}{\sum_{\ell = 1}^N \wgt{k}{\ell} \adj{k}{\ell}} \right)_{j = 1}^N \right) \eqsp.
\end{equation}
There are however many alternative ways of performing selection; \eg, one may set $\nbr{k}{i} \define \lfloor N \wgt{k}{i} \adj{k}{i} / \sum_{\ell = 1}^N \wgt{k}{\ell} \adj{k}{\ell} \rfloor + H_k^i$ with
\begin{multline} \label{eq:deter:plus:res:selection}
(H_k^1, \ldots, H_k^N) \\ \sim \operatorname{Mult}\left( \sum_{i = 1}^N \left \langle \frac{N \wgt{k}{i} \adj{k}{i}}{\sum_{\ell = 1}^N \wgt{k}{\ell} \adj{k}{\ell}} \right \rangle , \left( \frac{ \langle N \wgt{k}{i} \adj{k}{i} / \sum_{\ell = 1}^N \wgt{k}{\ell} \adj{k}{\ell} \rangle}{\sum_{j = 1}^N \langle N \wgt{k}{j} \adj{k}{j} / \sum_{\ell = 1}^N \wgt{k}{\ell} \adj{k}{\ell} \rangle} \right)_{i = 1}^N       \right) \eqsp,
\end{multline}
where $\lfloor x \rfloor$ denotes the integer part of a real number $x$ and $\langle x \rangle \define x - \lfloor x \rfloor$. In this selection schedule, which was proposed by \citet{liu:chen:1995} under the name \emph{deterministic plus residual multinomial resampling}, index $i$ is first copied $\lfloor N \wgt{k}{i} \adj{k}{i} / \sum_{\ell = 1}^N \wgt{k}{\ell} \adj{k}{\ell} \rfloor$ times; the remaining $\sum_{i = 1}^N \langle N \wgt{k}{i} \adj{k}{i} / \sum_{\ell = 1}^N \wgt{k}{\ell} \adj{k}{\ell} \rangle$ indices are hereafter drawn multinomially with respect to weights proportional to the residuals $(\langle N \wgt{k}{i} \adj{k}{i} / \sum_{\ell = 1}^N \wgt{k}{\ell} \adj{k}{\ell} \rangle)_{i = 1}^N$. All theoretical results obtained in the following will hold for both the selection schedules \eqref{eq:mult:selection} and \eqref{eq:deter:plus:res:selection}. In addition, our results are easily extended to selection schemes based on \emph{Poisson, binomial, and Bernoulli branching} \citep[see][for a theoretical analysis of these algorithms]{douc:moulines:2008}; however, since the number of drawn indices are random in this case, we omit these results for brevity.

\subsubsection{Convergence of the \GPEAPS}

We will describe the convergence, as $N$ tends to infinity, of the self-normalised Monte Carlo approximations formed by weighted particle samples returned by Algorithm~\ref{alg:SMC:smoothing} using the concept of \emph{consistency} \citep[adopted from][]{douc:moulines:2008} defined in the following. Let $(\partspace, \alg(\partspace))$ denote some given state space and $(\parti{}{i},\wgt{}{i})_{i = 1}^N$ a $\partspace$-valued particle sample.

\begin{definition}  \label{def:consistency}
    A weighted sample $(\parti{}{i},\wgt{}{i})_{i = 1}^N$ is \emph{consistent} for a probability measure $\mu$ and a set $\cset \subseteq \Lp{1}(\partspace, \mu)$ if, as $N \rightarrow \infty$,
    \begin{equation} \label{eq:consistencyconvergence}
        \wgtsum{}^{-1} \sum_{i=1}^N \wgt{}{i} f(\parti{}{i}) \ConvP \mu(f) \eqsp, \quad \text{for\ all\ } f \in \cset \eqsp,
    \end{equation}
    and, additionally,
    \begin{equation}\label{eq:weightconvergence}
        \wgtsum{}^{-1} \max_{1 \leq i \leq N} \wgt{}{i} \ConvP 0 \eqsp.
    \end{equation}
\end{definition}

The following assumption is mild (in fact, minimal) but essential when establishing consistency of the \GPEAPS\ scheme.

\begin{hyp}{adj:in:L1}
    For all $0 \leq k \leq n$, $\adjfunc{k} \in \Lp{1}(\Xset^{k + 1}, \sm{k})$ and $\uk{k}(\cdot, \Xset) \in \Lp{1}(\Xset, \sm{k|k})$.
\end{hyp}

\begin{proposition} \label{prop:cons:GPEAPS}
    Assume \refhyps{Beskos:C0}{adj:in:L1} and that the initial sample $(\parti{0}{i}, \wgt{0}{i})_{i = 1}^N$ is consistent for $(\sm{0}, \Lp{1}(\Xset, \sm{0}))$. Then, for all $1 \leq k \leq n$, each sample $(\parti{0:k|k}{i}, \wgt{k}{i})_{i = 1}^N$ produced by Algorithm~\ref{alg:SMC:smoothing} is consistent for $(\sm{k}, \Lp{1}(\Xset^{k + 1}, \sm{k}))$. The same is true when the multinomial selection schedule \eqref{eq:mult:selection} is replaced by deterministic plus residual multinomial selection \eqref{eq:deter:plus:res:selection}.
\end{proposition}

The proof of Proposition~\ref{prop:cons:GPEAPS} is postponed to Appendix~\ref{section:proof:GPEAPS:conv}.

\subsection{Fixed-lag smoothing}
\label{section:fixed-lag}

Unfortunately, it has been observed by several authors that using standard SMC methods in the smoothing mode may be unreliable for larger observation sample sizes $n$, since resampling systematically the particles degenerates the particle paths. Indeed, when $k \ll n$, most (or possibly all) marginal particles $(\parti{k|n}{i})_{i = 1}^N$ will coincide, resulting in a significant Monte Carlo error when estimating any expectation of $X_k$ given $Y_{0:n}$ using the produced particles. Especially, returning to the problem of estimating the intermediate quantity $\IMQ_n$ in \eqref{Eq:IMQ}, for any type of \emph{additive functional} $t(x_{0:n}) \define \sum_{k = 0}^{n - 1} s_k(x_{k:k + 1})$, $(s_k)_{k = 0}^{n - 1}$ being a set of functions (cf. the two terms of \eqref{Eq:IMQ}), we may expect that the estimator
\begin{equation} \label{eq:crude:estimator}
    (\wgtsum{n})^{- 1} \sum_{k = 0}^{n - 1} \sum_{i = 1}^N \wgt{n}{i} s_k( \parti{k:k + 1|n}{i})
\end{equation}
of $\ee [ t(X_{0:n}) | Y_{0:n} ]$ is poor when $n$ is large. To compensate for this degeneracy the particle sample size $N$ has to be increased drastically, yielding a computationally inefficient algorithm.

On the other hand, since we may expect that remote observations are only weakly dependent, it should hold that, for a large enough integer $\Delta_n$,
\begin{equation*} \label{eq:forgetting:approx}
    \ee \left[ s_k(X_{k:k + 1}) | Y_{0:n} \right] \approx \ee \left[ s_k(X_{k:k + 1}) | Y_{0:k(\Delta_n)} \right] \eqsp,
\end{equation*}
where $k(\Delta_n) \define \min \{ k + \Delta_n, n\}$, yielding
\begin{equation} \label{eq:forgetting:approx-2}
    \ee [ t(X_{0:n}) | Y_{0:n} ] = \sum_{k = 0}^{n - 1} \ee \left[ s_k(X_{k:k + 1}) | Y_{0:n} \right ] \approx  \sum_{k=0}^{n-1} \ee \left[ s_k(X_{k:k+1}) | Y_{0:k(\Delta_n)} \right] \eqsp.
\end{equation}
Thus, as long as the approximation \eqref{eq:forgetting:approx-2} is relatively precise for a $\Delta_n$ which is smaller than the average particle trajectory collapsing time, \ie\ most marginal particles $(\parti{k|k(\Delta_n)}{i})_{i = 1}^N$ are different for all $k$, we should replace \eqref{eq:crude:estimator} by the estimator
\begin{equation} \label{eq:lag:estimator}
    \sum_{k = 0}^{n - 1} \left( \wgtsum{k(\Delta_n)} \right)^{-1} \sum_{i = 1}^N \wgt{k(\Delta_n)}{i} s_k\left( \parti{k:k + 1|k(\Delta_n)}{i} \right) \eqsp.
\end{equation}
The lag-based approximation \eqref{eq:lag:estimator} may be computed recursively in a \emph{single sweep} of the data with only limited computer data storage demands, and computing \eqref{eq:lag:estimator} is clearly not more computationally demanding than computing \eqref{eq:crude:estimator} (having $O(nM)$ complexity); see \citet{olsson:cappe:douc:moulines:2008} for details. Finally, using \eqref{eq:lag:estimator} in conjunction with the kernel $\lqkern_\theta$ for estimating $\log \hd_\theta$ gives us the following approximation of the intermediate quantity $\IMQ_n(\theta; \theta')$:
\begin{equation} \label{eq:def:Q:est}
    \IMQ_n^N (\theta; \theta') \define \sum_{k = 0}^{n - 1} \left( \Omega_{k(\Delta_n)}^{N, \theta'} \right)^{ - 1} \sum_{i = 1}^N \wgt{k(\Delta_n)}{i, \theta'} s_k^{\bar{\alpha}}\left( \parti{k:k + 1|k(\Delta_n)}{i, \theta'} ; \theta \right) \eqsp,
\end{equation}
where, for $(x, x') \in \Xset^2$,
\[
    s_k^{\bar{\alpha}}(x, x' ; \theta)  \define \frac{1}{\bar{\alpha}} \sum_{\ell = 1}^{\bar{\alpha}} \logqdraw{x, x'}{\ell}{\theta} + \log \lk_\theta(x', Y_{k + 1})
\]
and
\[
    \logqdraw{x, x'}{1:\bar{\alpha}}{\theta} \sim \lqkern_\theta^{\varotimes \bar{\alpha}}(x, x', \cdot) \eqsp.
\]
In \eqref{eq:def:Q:est} we have added $\theta'$ as an index to the particles as well as the  associated weights to indicate that the particle system of the fixed-lag smoother is evolved under the dynamics determined by the initial parameter value.

\subsubsection{Convergence of the intermediate quantity}
\label{section:theoretical:results}
Under weak assumptions on the functions $\adjfunc{k}$, the kernels $\uk{k}$ and $\lqkern$, and the local likelihoods functions $\log \lk_\theta(\cdot, Y_k)$ one may establish the convergence of the approximate intermediate quantity \eqref{eq:def:Q:est}. Thus, define, for a given lag $\Delta_n$ and parameters $(\theta, \theta')$, the bias
\begin{multline} \label{eq:def:bias}
    b_n(\Delta_n, \theta, \theta') \define \sum_{k = 0}^{n - 1} \int s_k(x_{k:k + 1}, \theta) \, \sm{k(\Delta_n)}(\dd x_{k:k + 1} , \theta') \\
    - \sum_{k = 0}^{n - 1} \int s_k(x_{k:k + 1}, \theta) \, \sm{n}(\dd x_{k:k + 1} , \theta')
\end{multline}
imposed by the fixed lag. We then have the following result, which is the main result of this section.

\begin{theorem} \label{th:consistency}
Assume \refhyps{Beskos:C0}{Beskos:C2}. Let $n \geq 0$, $(\theta, \theta') \in \Theta^2$, and $(\Delta_n, \alpha, \bar{\alpha}) \in \mathbb{N}^3$. Suppose that \refhyp{adj:in:L1} holds  for $\adjfunc{k}(\cdot ; \theta')$, $\uk{k}(\cdot; \theta')$, and $\sm{k}(\cdot; \theta')$ and that the initial sample $(\parti{0}{i, \theta'}, \wgt{0}{i, \theta'})_{i = 1}^N$ is consistent for $(\sm{0}(\cdot; \theta'), \Lp{1}(\sm{0}(\cdot; \theta'), \Xset))$. Moreover, assume that the mappings $x_{0:k(\Delta_n)} \mapsto \log g_\theta(x_k, Y_k)$, $0 \leq k \leq n$, and $x_{0:k(\Delta_n)} \mapsto \int |v| \lqkern_\theta(x_k, x_{k + 1}, \dd v)$, $0 \leq k < n$, belong to $\Lp{1}(\sm{k(\Delta_n)}(\cdot; \theta'), \Xset^{k(\Delta_n) + 1})$. Then, as $N \rightarrow \infty$,
\begin{equation*}
    \IMQ^N_n(\theta, \theta') \ConvP \IMQ_n(\theta, \theta') + b_n(\Delta_n, \theta, \theta') \eqsp,
\end{equation*}
where the bias $b_n$ is defined in \eqref{eq:def:bias}.
\end{theorem}
The proof is given in Appendix~\ref{section:proof:Q:conv}.

The bias term $b_n$, which was treated by \cite{olsson:cappe:douc:moulines:2008}, is controlled by the speed with which the hidden chain $(X_k)_{k \geq 0}$ forgets its initial distribution when evolving \emph{conditionally} on the observations. Indeed, when the state space $\Xset$ is compact it can be shown \cite[see][for details]{olsson:cappe:douc:moulines:2008} that $b_n$ is $\mathcal{O}(n \rho^{\Delta_n})$, where $0 < \rho < 1$ is the \emph{uniform} (with respect to observation records $Y_{0:n}$ as well as initial distributions $\init$) mixing coefficient of the conditional chain. From this we deduce that the lag $\Delta_n$ should be increased with $n$ at the minimum rate $c \log n$, $c > -1/\log \rho$ in order to keep the bias suppressed. Increasing $\Delta_n$ faster eliminates the bias and increases the variance of the approximation; see again \citet{olsson:cappe:douc:moulines:2008} for a detailed study of these issues. Since a similar forgetting property holds also in the case of a non-compact state space $\Xset$ \citep{douc:fort:moulines:priouret:2007}, the same arguments can be applied for very general models; however, the analysis of the general case is significantly more involved, since the mixing coefficient is neither uniform with respect to observation records nor initial distributions $\init$ in this case.

Remarkably, the convergence result in Theorem~\ref{th:consistency} holds for \emph{any fixed sample sizes} $(\alpha, \bar{\alpha})$. In particular, nothing prevents us from letting $\alpha = \bar{\alpha} = 1$, yielding a computationally very efficient algorithm; this is the choice of Section~\ref{section:simulation:study}.

\subsection{Forward-filtering backward-smoothing}
\label{section:FFBS}
Even though naive SMC implementations generally fail to estimate joint smoothing distributions efficiently, they can, as discussed above, be successfully used for estimating the marginal filter distributions (corresponding to $k = n$ in the discussion of Section \ref{section:fixed-lag}). Nevertheless, any joint smoothing distribution may be expressed in terms of marginal filter distributions via the so-called \emph{forward-filtering backward-smoothing decomposition}. Indeed, for any probability measure $\eta$ on $(\Xset, \Xalg)$, define the \emph{reverse kernel}
\begin{equation} \label{eq:backward:kernel:density:form}
    \revM{\eta}(x', A ; \theta) \define \frac{\int_A \hd_\theta(x, x') \,
    \eta(\dd x)}{\int \hd_\theta(x, x') \, \eta(\dd x)} \eqsp,
\end{equation}
where $A \in \Xalg$ and $x' \in \Xset$. The definition \eqref{eq:backward:kernel:density:form} is valid only when $x'$ belongs to the subset of $\Xset$ where the denominator is nonzero; outside this set we may let $\revM{\eta}$ take arbitrary values. It can now be shown that \citep[see \eg][Corollary~3.3.8]{cappe:moulines:ryden:2005}
\begin{equation} \label{eq:smoothing:backw:decomposition}
    \sm{n}(A ; \theta) = \idotsint_A \sm{n|n}(\dd x_n ; \theta) \, \prod_{k = 0}^{n - 1}
    \revM{\sm{k|k}}(x_{k + 1}, \dd x_k ; \theta) \eqsp,
\end{equation}
for $A \in \Xalg^{\varotimes (n + 1)}$. Using the Markovian structure of the decomposition above, a trajectory $\partdraw{0:n}$ can be simulated from $\sm{n}(\cdot ; \theta)$ by, firstly, computing recursively (via \eqref{eq:smoothing:recursion}) the filter distributions $(\sm{k|k}(\cdot ; \theta))_{k = 0}^n$ and, secondly, simulating $\partdraw{n}$ from $\sm{n|n}(\cdot ; \theta)$ and hereafter, recursively for $k = n - 1, n - 1, \ldots, 0$, $\partdraw{k}$ from $\revM{\sm{k|k}}(\partdraw{k + 1}, \cdot ; \theta)$. This scheme will in the following be referred to as \emph{forward-filtering backward-simulation} (FFBS), and we refer again to \citet{cappe:moulines:ryden:2005} for a detailed treatment.

In general we lack closed-form expressions of the filter distributions, but may estimate these efficiently using Algorithm \ref{alg:SMC:smoothing}. Hence, following \citet{doucet:godsill:andrieu:2000}, a non-degenerate particle estimate of $\sm{0:n}(\cdot ; \theta)$ can be obtained by replacing, in the decomposition \eqref{eq:smoothing:backw:decomposition}, $\sm{n|n}$ by the empirical measure $\partsm{n|n}$ and the reverse kernels $\revM{\sm{k|k}}(x_{k + 1}, \dd x_k ; \theta)$ by
\begin{equation} \label{eq:reverse:approx}
    \revM{\sm{k|k}^N}(x_{k + 1}, \dd x_k ; \theta) = \sum_{i = 1}^N \frac{\wgt{k}{i} \hd_\theta(\parti{k|k}{i}, x_{k + 1})}{\sum_{\ell = 1}^N \wgt{k}{\ell} \hd_\theta(\parti{k|k}{\ell}, x_{k + 1})} \delta_{\parti{k|k}{i}}(\dd x_k) \eqsp.
\end{equation}
Note that a draw according to $\revM{\sm{k|k}^N}(x_{k + 1}, \cdot ; \theta)$ consists of selecting position $\parti{k|k}{i}$ with probability proportional $\wgt{k}{i} \hd_\theta(\parti{k|k}{i}, x_{k + 1}) / \sum_{\ell = 1}^N \wgt{k}{\ell} \hd_\theta(\parti{k|k}{\ell}, x_{k + 1})$. In the case of PODs, a closed-form expression of $\hd_\theta$ is in general missing, and we thus replace each number $\hd_\theta(\parti{k|k}{i}, x_{k + 1})$ by a draw $\qdraw{\parti{k|k}{i}, x_{k + 1}}{}{\theta}$ from the GPE $\qkern_\theta(\parti{k|k}{i}, x_{k + 1}, \cdot)$. This gives us the following algorithm for simulating a trajectory $\partdraw{0:n}$ that is approximately distributed according to $\sm{n}$.

\bigskip

\begin{algorithm}{GPE-FFBS}{\label{alg:FFSB}
    \qcomment{GPE-based particle FFBS}
    \qinput{$(\proposal{k})_{k = 0}^{n - 1}$}}
    run Algorithm~\ref{alg:SMC:smoothing} to obtain $(\partsm{k|k}(\cdot ; \theta))_{k = 0}^n$; \\
    simulate $\partdraw{n} \sim \partsm{n|n}(\cdot ; \theta)$; \\
    \qfor $k \qlet n - 1$ \qto $0$ \\
    \qfor $i \qlet 1$ \qto $N$ \\
    simulate $\qdraw{\parti{k|k}{i}, \partdraw{k + 1}}{}{\theta} \sim \qkern_\theta(\parti{k|k}{i}, \partdraw{k + 1}, \cdot)$; \qrof \\
    simulate $\iota_k \sim (\wgt{k}{i} \qdraw{\parti{k|k}{i}, \partdraw{k + 1}}{}{\theta} / \sum_{\ell = 1}^N \wgt{k}{\ell} \qdraw{\parti{k|k}{\ell}, \partdraw{k}}{}{\theta})_{i = 1}^N$; \\
    set $\partdraw{k} \qlet \parti{k|k}{\iota_k}$ \qrof \\
    \qreturn $\partdraw{0:n} = (\partdraw{0}, \ldots, \partdraw{n})$.
\end{algorithm}

\bigskip

Algorithm~\ref{alg:FFSB} avoids the problem of degeneracy of the genealogical tree \emph{without} any implicit assumption on geometrical ergodicity of the conditional hidden chain. On the other hand, simulating a single trajectory according to Algorithm~\ref{alg:FFSB} involves $\mathcal{O}(N)$ operations, implying an overall computational cost of order $\mathcal{O}(N^2)$ for producing a sample of size $N$. Recently, \citet{douc:garivier:moulines:olsson:2009} showed how the overall computational cost of the particle-based FFBS can be reduced to $\mathcal{O}(N)$ by means of accept-reject-methods; however, it is not straightforward to adapt this approach to our framework, since one for general PODs cannot find an upper bound on the transition density of the hidden chain. For models with forgetting properties, Algorithm~\ref{alg:FFSB} should be outperformed by the fixed-lag smoother because of the quadratic complexity of the former scheme (see the coming section for examples); the FFBS should thus be seen as a generic and alternative solution in cases of poor mixing.  

\section{Simulation study}
\label{section:simulation:study}
In this section, the proposed methods are illustrated on two simulated examples, consisting of noisy observations of the models treated by \cite{beskos:papaspiliopoulos:roberts:fearnhead:2006} and \cite{beskos:papaspiliopoulos:roberts:2008}. In both examples we let, for simplicity, the measurement noise variance $\sigma_\epsilon$ be known and set to $0.1$ and assume equidistant measurements with $t_{k+1} - t_k = 1$ for all $k \geq 0$.  We use consequently $\alpha = \ba = 1$. The approximate intermediate quantity $\IMQ_n^N$ is maximised using the \emph{Nelder-Mead simplex algorithm} as implemented in MATLAB's {\verb fminsearch }-command. In order to obtain convergence of the parameter sequence returned by the Monte Carlo EM-algorithm, it is necessary to decrease, at each iteration, the bias of the particle approximation by increasing the number of particles with the iteration index. We thus follow the recommendations of \citet{fort:moulines:2003} and increase the particle sample size as the square root of the iteration number, with an initial size of 100 particles. A detailed discussion on the effect of the lag size on the quality of the final parameter estimates is given in \citet{olsson:cappe:douc:moulines:2008}; thus, we do not repeat this discussion here and stick consequently to the recommendation of increasing the lag logarithmically with the size of the observation record.

\subsection{Log-growth model}
In the first example we estimate, from simulated data, the parameters of a partially observed version of the \emph{log-growth model} discussed by \cite{beskos:papaspiliopoulos:roberts:fearnhead:2006}. The model is specified by the following system of equations:
\begin{equation} \label{eq:def:log:growth:model}
\begin{split}
    \dd X_t & = \kappa X_t(1 - X_t/\Lambda) \, \dd t + \sigma X_t \, \dd W_t \eqsp, \\
    Y_k & =  X_{t_k} + \sigma_\epsilon \epsilon_k \eqsp,
\end{split}
\end{equation}
where $(\epsilon_k)_{k \geq 0}$ are mutually independent, standard normal-distributed random variables. The noise sequence is supposed to be independent also from $W$. Applying It$\hat{\mathrm{o}}$'s formula to the transformation $\tilde{X}_t = \eta(X_t, \sigma)$, with $\eta(x,\sigma) \define -\log(x)/\sigma$, yields
\begin{equation} \label{eq:def:log:growth:model:state:transformed}
\dd \tilde{X}_t = \alpha(\tilde{X}_t) + \dd W_t \eqsp, %\kappa X_t(1 - X_t/\Lambda) \, \dd t + \sigma X_t \, \dd W_t \eqsp, \\
\end{equation}
where $\alpha(x) \define \sigma / 2 - \kappa / \sigma + \kappa / (\sigma \lambda) \exp(- \sigma x)$. Since $\alpha$ is bounded from above, we are only required to simulate the minimum of the Brownian path and let $\Lower{\alpha}$ be $\alpha$ evaluated at this minimum; see Section~\ref{section:appendix:EA} for the meaning of $\Lower{\alpha}$. The minimum of the Brownian bridge has a known law, and given the minimum, the bridge can be constructed retrospectively using Bessel bridges \citep[see][]{beskos:papaspiliopoulos:roberts:fearnhead:2006}.
Our aim is to estimate the unknown parameters $\theta \define (\kappa, \Lambda, \sigma )$ given a record $Y_{0:1000}$ of observations. The observation set was obtained through simulation under the parameters $\theta^* = (0.1, 1000, 0.1 )$. When computing the approximate intermediate quantity $\IMQ_n^N$, the random weight fixed-lag smoother used the lag $\Delta_n = 40$ and the proposal
\begin{equation} \label{eq:proposal:ex:2}
    \proposal{k}(x, A) = \frac{1}{\sigma x} \int_A t( \{ x' - \kappa x (1 - x/\Lambda) \} / \{ \sigma x \} ; 4) \, \dd x' \eqsp,
\end{equation}
where $t(\cdot; n)$ denotes the density of the student's $t$-distribution with $n$ degrees of freedom. Further the adjustment multiplier weights are set to $1$. The proposal \eqref{eq:proposal:ex:2} is obtained by discretising the hidden dynamics using the Euler scheme. We set $\alpha = \ba = 1$. The EM output is presented in Figure~\ref{fig:LogGrowth:params}.

\begin{figure} \label{fig:LogGrowth:params}
    \centering
    \includegraphics[angle=0,width=0.75\textwidth]{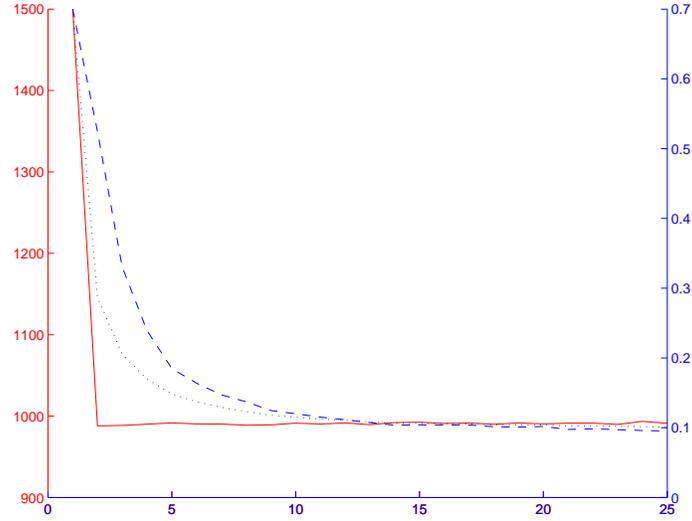}
    \caption{Convergence of $\Lambda$ (solid, left y-axis), $\kappa$ (dashed, right y-axis), and $\sigma$ (dotted, right y-axis) using the fixed-lag smoother with lag $40$.}
\end{figure}

For comparison, the estimation problem of the log-growth model was also solved using the GPE-based particle FFBS in Section~\ref{section:FFBS}. The setup was the same as for the fixed-lag smoother, but due to the significant higher computational cost of the FFBS scheme (recall Section~\ref{section:FFBS}) the number of observations was reduced to $100$. For the FFBS-based procedure, the GPE needs to be evaluated $N + 1$ times per particle and time step, \ie, once in the forward filtering pass and $N$ times in the backward simulation sweep, compared to only once for the fixed-lag smoother. 
%Besides this additional computational cost, the significantly larger amount of GPE evaluations needed for producing the FFBS-based estimates increases the variance relatively the fixed-lag smoother.

The output of the EM learning curves obtained using the GPE-based particle FFBS is presented in Figure~\ref{fig:LogGrowth:params:FFBSi}.

\begin{figure} \label{fig:LogGrowth:params:FFBSi}
    \centering
    \includegraphics[angle=0,width=0.75\textwidth]{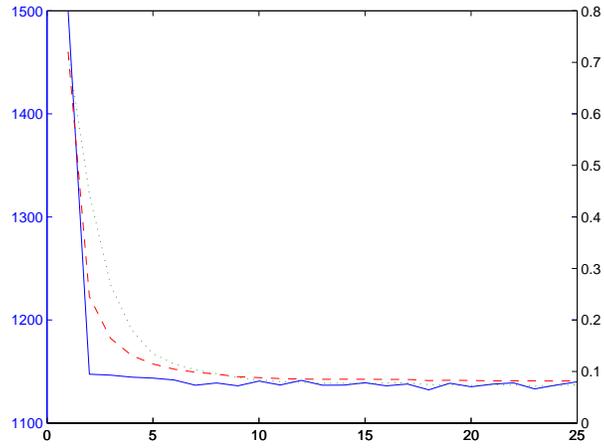}
    \caption{Convergence of $\Lambda$ (solid, left y-axis), $\kappa$ (dashed, right y-axis), and $\sigma$ (dotted, right y-axis) using the GPE-based particle FFBS on $100$ observations.}
\end{figure}

\subsection{Genetics diffusion model}
In a second example we estimate, again from simulated data, the parameters of a partially observed version of the \emph{genetics diffusion model} presented in \cite{book:kloeden:platen} and discussed by \cite{beskos:papaspiliopoulos:roberts:2008}. The model is given by
\begin{equation} \label{eq:def:log:growth:model}
\begin{split}
    \dd V_t & = (\mu + \nu V_t )\, \dd t + \sigma V_t (1 - V_t) \, \dd W_t \eqsp, \\
    Y_{t_k} & =  V_{t_k} + \sigma_\epsilon \epsilon_k \eqsp,
\end{split}
\end{equation}
where the sequence $(\epsilon_k)_{k \geq 0}$ is as in the previous example. Applying It$\hat{\mathrm{o}}$'s formula to the transformation $\tilde{X}_t = \eta(V_t, \sigma)$, where $\eta(v, \sigma) \define (\log(v) - \log(1 - v)) / \sigma$, allows for using the GPE for estimating the transition density of the latent process. In this case, the drift function $\alpha$ of the transformed process becomes more involved than in the previous example, and it is neither bounded from above nor below. Thus, we have to draw both $\Lower{\alpha}$ and $\Upper{\alpha}$ and a Brownian bridge $(\tilde{W}_s)_{s = 0}^t$ such that $\Lower{\alpha} \leq \alpha(\tilde{W}_s) \leq \Upper{\alpha}$ for all $0 \leq s \leq t$; see Section~\ref{section:appendix:EA} for a justification of this. For this purpose we apply the method proposed in \cite{beskos:papaspiliopoulos:roberts:2008}, which involves sampling first a maximum $\Upper{\operatorname{id}}$ and a minimum $\Lower{\operatorname{id}}$, and then a Brownian bridge such that $\Lower{\operatorname{id}} \leq \tilde{W}_s \leq \Upper{\operatorname{id}}$ for all $0 \leq s \leq t$. Since a linear transformation of a Brownian bridge is still a Brownian bridge, it suffices to consider the case when the path $(\tilde{W}_s)_{s = 0}^t$ is conditioned to start and end in zero. Sampling a lower and upper bound can then be done by using rejection sampling in the following way: let $(a_i)_{i \geq 0}$ with $a_0 = 0$ be an increasing sequence and consider the intervals $(-a_i, a_i]$. Since the probability that a Brownian bridge stays in a specific interval $[-K, K]$ has a known expression (having the form of an infinite series), it is possible to calculate the probability that it is contained in $(-a_i, a_i]$ but not in $(-a_{i-1}, a_{i-1}]$; this means that either its maximum is contained in $(a_{i-1}, a_i]$ or its minimum is contained in $(-a_i, -a_{i-1}]$ or both. Thus, we first propose an interval $(a_{i-1}, a_i]$; given this interval, we then propose, with probability $1/2$, a maximum conditioned to belong to $(a_{i-1}, a_i]$, otherwise a minimum in $(-a_i, -a_{i-1}]$. Since the distributions of the maximum and minimum are known on closed-form, this is easily done. Next, we propose a Brownian bridge by decomposing around the proposed maximum (minimum) as in the previous example. The resulting path $(\tilde{W}_s)_{s = 0}^t$ is accepted, with a probability depending on the path in question, only if it remains in the interval; see \cite{beskos:papaspiliopoulos:roberts:2008} for details. Finally, we set $\tilde{W}^{\pm}_\alpha \define \alpha(\tilde{W}^{\pm}_{\operatorname{id}})$.

Again we attempt to estimate the unknown parameters $\theta \define (\mu, \nu, \sigma )$ given a record $Y_{0:1000}$ of observations obtained through simulation under the parameters $\theta^* = (0.05, 0.1, 1 )$. When computing the approximate intermediate quantity $\IMQ_n^N$, the random weight fixed-lag smoother used the lag $\Delta_n = 20$. Since the state space $\rr(0,1)$ is compact, we propose the particles by simply drawing uniforms over $(0,1)$. We set $\alpha = \ba = 1$. The EM output in presented in Figure~\ref{fig:GenSDE:params}.

\begin{figure}\label{fig:GenSDE:params}
\centering
\includegraphics[angle=0,width=0.75\textwidth]{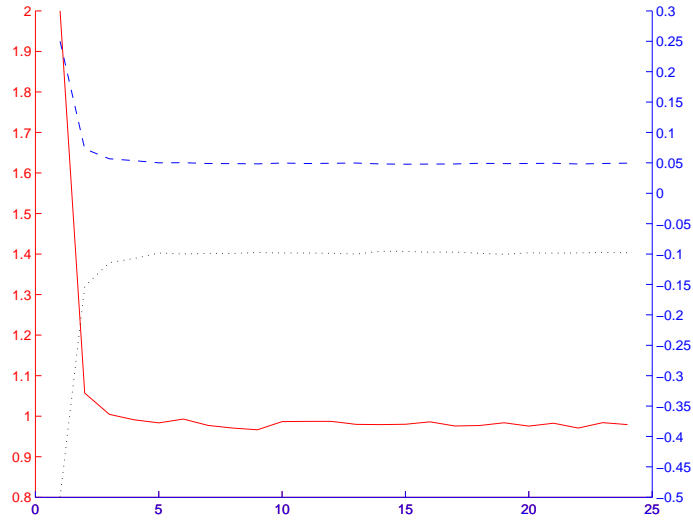}
\caption{Convergence of $\sigma$ (dotted, left y-axis), $\nu$ (dashed, right y-axis) and $\mu$ (dotted, right y-axis) .}
\end{figure}
\vspace{5mm}

\section{Conclusion}
\label{section:conclusion}
Parameter inference in general discretely and partially observed diffusion processes is an inherently difficult problem due to the lack of closed-form transition densities of the hidden Markov chain. Assuming the possibility of simulating exactly transitions of the latent diffusion process, it is possible to produce pointwise and consistent estimates of the likelihood function using the standard bootstrap particle filter, in which the particles are assigned importance weights determined completely by the known local likelihood function. In such a framework, the likelihood surface can be explored using \eg\ grid-based methods \citep{olsson:ryden:2005}. \citet{ionides:bahdra:king:2009} use the bootstrap particle filter for computing pointwise approximations of the score function and locate the maximum likelihood estimate by means of stochastic approximation. However, simulating exactly transitions of a diffusion process is in general infeasible and we are most often referred to discretisation-based methods such as the Euler scheme, imposing a nontrivially controlled bias of the final parameter estimates. Moreover, mutating blindly, as in the bootstrap particle filter, the particles without incorporating, in the proposal kernel, the information provided by the observations will in general lead to serious degeneracy of the particle weights, especially for models where the observations are informative.

Thus, in the present paper we proposed an alternative, EM-based method for estimating unknown parameters of PODs. The method combines recent approaches for estimating efficiently the joint smoothing distribution in hidden Markov models with recently proposed techniques for estimating, without bias, transition densities of a large class of diffusion processes via GPEs \citep{beskos:papaspiliopoulos:roberts:2008}. Interestingly, the GPE provides a way of producing unbiased estimates of the transition densities \emph{simultaneously} for all parameter values; this is critical when carrying through the maximisation-step of the EM-algorithm. For models having forgetting properties, the degeneracy of the particle trajectories can be efficiently avoided by means of fixed-lag smoothing \citep{kitagawa:1998,olsson:cappe:douc:moulines:2008}. The decrease of variance gained by the fixed-lag approximation is obtained at the cost of a bias; the bias is however easily controlled by increasing logarithmically the size of the lag with the size of the observation record, yielding an algorithm of $\mathcal{O}(N)$ computational complexity. We provide a detailed study of the convergence of the GPE-based particle smoother as well as the full intermediate quantity of EM. The results are obtained under, what we believe, minimal assumptions and may, since we analyse separately the GPE-based mutation step (Lemma~\ref{lem:particleconsistency}), be extended to any selection schedule for which consistency has been established in the literature. In this way, our GPEPS convergence results differ significantly from that presented in \cite{fearnhead:papaspiliopoulos:roberts:2008}. In the non-ergodic case, we proposed a method for sampling the joint smoothing distribution which is based on the forward-filtering backward-smoothing decomposition of the same. Basically, the method, which relies on an algorithm proposed by \citet{godsill:doucet:west:2004} and analysed further by \citet{douc:garivier:moulines:olsson:2009}, consists of a forward-filtering pass followed by a backward-simulation pass where trajectories are drawn according to approximations of the backward kernels obtained using the particle filter estimates obtained in the forward pass. During the two passes we replace, when needed, any evaluation of the diffusion process transition density by a draw from the GPE. At the end of the day, we obtain an $\mathcal{O}(N^2)$ algorithm that is significantly more costly than the fixed-lag smoother, but which avoids elegantly the problem of degeneracy of the genealogical tree of the particles. The methods were successfully demonstrated on two examples.

There exist alternative techniques, either Monte Carlo-based \citep[see \eg][]{pedersen:1995:theory} or based on basis expansions \citep{aitsahalia:2008}, for approximating the transition density. Nevertheless, none of these approaches produce unbiased estimates. The former is, while quite general, computationally very demanding and the latter is only valid for very short time intervals (recall that the performance of the GPE is independent of the size of the time grid). Sometimes more direct numerical approaches, such as solving the Fokker-Plank equations or taking the Fourier inverse of the characteristic function of the SDE, are possible; however, these methods often tend to be computationally expensive. Anyway, the theoretical results obtained by us presume only unbiasedness of the transition density estimator, and thus other approximation schemes may be applicable within our framework. 

\appendix

\section{Proofs}
\label{section:proofs}
The proofs of Proposition~\ref{prop:cons:GPEAPS} and Theorem~\ref{th:consistency} rely on recent results on limit theorems for weighted samples obtained by \cite{douc:moulines:2008}. Since we in this section deal exclusively with asymptotic properties of the sample as the sample size tends to infinity, we let, when not specified differently, the limit notation $\rightarrow$ refer to an \emph{increasing number} $N$ \emph{of particles} only. In addition, we let also the particles and the associated weights be indexed by $N$ for clearness. The following kernel notation will be useful in the following: Let $\mu$ be a measure on $(\partspace, \alg(\partspace))$, $f$ a measurable function on $(\partspacenew, \alg(\partspacenew))$, and $K$ a kernel from $(\partspace, \alg(\partspace))$ to $(\partspacenew, \alg(\partspacenew))$; then we set
\[
    \mu K(A) \define \int \mu(\ud \xi) \, K(\xi, A)
\]
and
\[
    K(\xi, f) \define \int f(\tilde{\xi}) \, K(\xi, \ud \tilde{\xi}) \eqsp.
\]
The following definition specifies the structure that we want any class of estimand functions to have.
\begin{definition}  \label{def:proper:set}
    A set $\cset$ of measurable functions on $\partspace$ is \emph{proper} if the following holds.
    \begin{itemize}
        \item[\emph{(i)}] $\cset$ is a linear space; that is, if $f$ and $g$ belong to $\cset$ and $(\alpha, \beta) \in \rr^2$, then $\alpha f + \beta g \in \cset$;
        \item[\emph{(ii)}] if $g \in \cset$ and $f$ is measurable with $|f| \leq |g|$, then $f \in \cset$;
        \item[\emph{(iii)}] for all $c \in \rr$, the constant function $\xi \mapsto c$ belongs to $\cset$.
    \end{itemize}
\end{definition}
We will frequently make use of the following lemma obtained by \citet[][]{douc:moulines:2008}. Let $(\Omega, \partalg, \mathbb{P})$ be a probability space and $(\partalg_{N, i})_{i = 0}^N$, $N \geq 1$, a triangular array of sub-$\sigma$-fields of $\partalg$ such that $\partalg_{N, i - 1} \subseteq \partalg_{N, i}$ for all $1 \leq i \leq N$ and $N \geq 1$. In addition, let $(\U{i})_{i = 1}^N$, $N \geq 1$, be a triangular array of random variables such that each $\U{i}$ is $\partalg_{N, i}$-measurable.
%%%%%%%%%%%%%%%%%%%%%%%%%%%%%%
%% Douc and Moulines, 2008 %%%
%%%%%%%%%%%%%%%%%%%%%%%%%%%%%%
\begin{theorem}[\cite{douc:moulines:2008}]  \label{th:DoucMoulines} %\textbf{Douc and Moulines Theorem 11}\label{th:DoucMoulines}
Assume that $\ee \left [|U_{N, j}| \vert \partalg_{N,j-1}\right ] < \infty$, $\pp$-a.s., for all $N \geq 1$ and $1 \leq j \leq N$. Suppose that
\begin{itemize}
    \item[{\it (i)}] as $\lambda \rightarrow \infty$,
    \begin{equation}
        \sup_{N \geq 1} \pp \left ( \sum_{j = 1}^N \ee \left[ \left. | \U{j} | \right| \partalg_{N, j - 1} \right]  \geq \lambda \right) \longrightarrow 0 \eqsp;
    \end{equation}
    \item[{\it (ii)}] in addition, for all $\epsilon > 0$,
    \begin{equation}
        \sum_{j = 1}^N \ee \left[  \left. |\U{j}| \condexp | \U{j} | \geq \epsilon \right| \partalg_{N,j - 1} \right ] \ConvP 0
    \end{equation}
\end{itemize}
as $N \rightarrow \infty$. Then
\begin{equation*}
    \max_{1\leq i \leq N} \left| \sum_{j = 1}^i \U{j} - \sum_{j = 1}^i  \ee \left[ \left. \U{j} \right| \partalg_{N, j - 1} \right] \right| \ConvP 0 \eqsp.
\end{equation*}
\end{theorem}

\subsection{Proof of Proposition~\ref{prop:cons:GPEAPS}}
\label{section:proof:GPEAPS:conv}

Algorithm~\ref{alg:SMC:smoothing} is conveniently analysed within a more general framework of \emph{random weight mutation} (RWM).
Assume that we are given a $\partspace$-valued, weighted particle sample $(\parti{}{i}, \wgt{}{i})_{i = 1}^N$ which is consistent for some measure $\nu$ on $\alg(\partspace)$ and let $L$ be a finite transition kernel from $(\partspace, \alg(\partspace))$ to $(\partspacenew, \alg(\partspacenew))$. We wish to transform $(\parti{}{i}, \wgt{}{i})_{i = 1}^N$ into another sample $(\partinew{}{i}, \wgtnew{}{i})_{i = 1}^N$ targeting the measure
\begin{equation*}
    \mu(A) = \frac{\nu \ukvar(A)}{\nu \ukvar(\partspacenew)} \eqsp, \quad A \in \alg(\partspacenew) \eqsp,
\end{equation*}
by means of the RWM operation described below. The input parameters are: a proposal kernel $\proposal{}$ such that  $\proposal{}(\xi, \cdot)$ dominates $\ukvar(\xi, \cdot)$ for all $\xi \in \partspace$, a random weight kernel $\lemkern$ from $(\partspace \times \partspacenew, \alg(\partspace \times \partspacenew))$ to $(\rr^+, \alg(\rr^+))$ targeting $\dd \ukvar / \dd \proposal{}$ in the sense that, for all $(\xi, \tilde{\xi}) \in \partspace \times \partspacenew$,
\[
    \int v \, \lemkern(\xi, \tilde{\xi}, \dd v) = \frac{\dd \ukvar(\xi, \cdot)}{\dd \proposal{}(\xi, \cdot)}(\tilde{\xi}) \eqsp,
\]
and, finally, a Monte Carlo sample size $\alpha \in \N$.
\bigskip
\begin{algorithm}{RWM}{\label{alg:SMC:rw-mutation}
    \qcomment{random weight mutation}
    \qinput{$(\parti{}{i}, \wgt{}{i})_{i = 1}^N$,
    $\proposal{}$, $\lemkern$, $\alpha$}}
    \qfor $i \qlet 1$ \qto $N$ \\
    \qdo simulate $\partinew{}{i} \sim \proposal{}(\parti{}{i},\cdot)$; \\
    simulate $\qdrawnew{\parti{}{i},\partinew{}{i}}{1:\alpha}{} \sim
    \lemkern^{\varotimes \alpha}(\parti{}{i},\partinew{}{i}, \cdot)$; \\
    $\wgtnew{}{i} \qlet \wgt{}{i} \alpha^{-1} \sum_{\ell = 1}^{\alpha} \qdrawnew{\parti{}{i},\partinew{}{i}}{\ell}{}$; \qrof \\
    \qreturn $(\partinew{}{i}, \wgtnew{}{i})_{i = 1}^N$.
\end{algorithm}
\bigskip
The sample $(\partinew{}{i}, \wgtnew{}{i})_{i = 1}^N$
returned by the algorithm is taken as an approximation of $\mu$.
In order to evaluate the quality of this sample, define the set
\begin{equation} \label{eq:def:C:tilde}
    \csetnew \define \left\{ f \in \mathsf{L}^1(\mu, \partspacenew) :
    \ukvar(\cdot, |f|) \in \cset \right\} \eqsp; %\iint |f(\tilde{\xi})| v \, \lemkern(\cdot, \tilde{\xi}, \dd v) \, \proposal{}(\cdot, \dd \tilde{\xi}) \in \cset \right\} \eqsp;
\end{equation}
then the following result stating consistency for weighted samples produced by Algorithm~\ref{alg:SMC:rw-mutation} is instrumental when establishing Proposition~\ref{th:consistency}.

%%%%%%%%%%%%%%%%%%%%%%%%%%%%%%%%%%%%%%%%%%%%%%%%%%%%%%
%%% Lemma: consistency RWM %%%%%%%%%%%%%%%%%%%%%%%%%%%
%%%%%%%%%%%%%%%%%%%%%%%%%%%%%%%%%%%%%%%%%%%%%%%%%%%%%%

\begin{lemma} \label{lem:particleconsistency}
    Assume the weighted sample $(\parti{}{i}, \wgt{}{i})_{i = 1}^N$ is consistent for $(\nu, \cset)$ and that the function $\ukvar(\cdot, \partspacenew)$ belongs to $\cset$. Then the set $\csetnew$ defined in \eqref{eq:def:C:tilde} and the weighted particle sample $(\partinew{}{i}, \wgtnew{}{i})_{i = 1}^N$ produced by Algorithm~\ref{alg:SMC:rw-mutation} are proper resp. $(\mu, \csetnew)$-consistent for any fixed $\alpha \in \N$.
\end{lemma}

%%% Proof %%%%%

\begin{proof}
Properness of the set $\csetnew$ is straightforwardly established: To check Property (i) in Definition~\ref{def:proper:set}, suppose that $f$ and $g$ belong to $\csetnew$ and let $(\alpha, \beta) \in \rr^2$; then
\begin{multline*}
    \iint |\alpha f(\tilde{\xi}) + \beta g(\tilde{\xi})| v \, \lemkern(\cdot, \tilde{\xi}, \dd v) \, \proposal{}(\cdot, \dd \tilde{\xi}) \\
    \leq |\alpha| \iint |f(\tilde{\xi})| v \, \lemkern(\cdot, \tilde{\xi}, \dd v) \, \proposal{}(\cdot, \dd \tilde{\xi}) \\
    + |\beta| \iint |g(\tilde{\xi})| v \, \lemkern(\cdot, \tilde{\xi}, \dd v) \, \proposal{}(\cdot, \dd \tilde{\xi}) \\
    = |\alpha| \ukvar(\cdot, |f|) + |\beta| \ukvar(\cdot, |g|) \eqsp,
\end{multline*}
where the function on the right hand side belongs to $\cset$ by construction of $\csetnew$ and the fact that $\cset$ is a linear space. That the integral on the left hand side belongs to $\cset$ is now a consequence of Property~(ii) in Definition~\ref{def:proper:set}. Properties (ii) and (iii) are checked in a similar manner.

To establish Condition~\eqref{eq:consistencyconvergence} in Definition~\ref{def:consistency} it is enough to show that, for all $f \in \csetnew$,
\begin{equation} \label{eq:key:limit-1}
    \wgtsum{}^{-1} \sum_{i = 1}^N \wgtnew{}{i} f(\partinew{}{i})
    \ConvP \nu \ukvar(f) \eqsp;
\end{equation}
indeed, since $\csetnew$ contains the unity mapping $\tilde{\xi} \mapsto 1$ (as $\csetnew$ is proper), \eqref{eq:key:limit-1} implies that
\begin{equation} \label{eq:weight:limit}
    \wgtsum{}^{-1} \sum_{i = 1}^N \wgtnew{}{i} \ConvP \nu L(\partspacenew) \eqsp,
\end{equation}
from which Condition~\eqref{eq:consistencyconvergence} in Definition~\ref{def:consistency}
follows by Slutsky's lemma. Thus, we define the triangular array  $\U{i} \define
\wgtnew{}{i} f(\partinew{}{i}) / \wgtsum{}$, $N \geq 1$, $1 \leq i \leq N$, and sub-$\sigma$-fields $\partalg_N \define \sigma\{ (\parti{}{i}, \wgt{}{i})_{i = 1}^N \}$, $N \geq 1$. We then get, by applying the tower property of conditional expectations and the consistency of the ancestor sample,
\begin{multline*}
    \sum_{i = 1}^N \ee \left[ \left. \U{i} \right| \partalg_N \right] \\
    = \wgtsum{}^{-1} \sum_{i = 1}^N \wgt{}{i} \ee \Bigg[ \Bigg.
    \ee \Bigg[ \Bigg. \alpha^{-1} \sum_{\ell = 1}^{\alpha}
    \qdrawnew{\parti{}{i},\partinew{}{i}}{\ell}{} \Bigg| \partinew{}{i},
    \partalg_N \Bigg] f(\partinew{}{i}) \Bigg| \partalg_N \Bigg] \\
    = \wgtsum{}^{-1} \sum_{i = 1}^N  \wgt{}{i}
    \int f(\tilde{\xi}) \int v \, \lemkern (\parti{}{i}, \tilde{\xi},\dd v) \,
    \proposal{} (\parti{}{i},\dd \tilde{\xi}) \\
    = \wgtsum{}^{-1} \sum_{i = 1}^N  \wgt{}{i} L(\parti{}{i}, f)
    \ConvP \nu \ukvar(f) \eqsp,
\end{multline*}
since $L(\cdot, f) \leq L(\cdot, |f|) \in \cset$. To show
that $\sum_{i = 1}^N \U{i}$ tends to $\sum_{i = 1}^N
\ee[\U{i} \vert \partalg_N]$ in probability, implying \eqref{eq:key:limit-1}, we apply
Theorem~\ref{th:DoucMoulines}. In order to
establish the first condition of that theorem we reuse the arguments above and use that $\ukvar(\cdot, |f|) \in \cset$, yielding the limit
\begin{equation*}
    \sum_{i = 1}^N \ee \left[ \left. |\U{i}| \right| \partalg_N \right]
    \ConvP \nu \ukvar(|f|) \eqsp.
\end{equation*}
Now, since convergence in probability implies
tightness, we conclude that Condition~(i) in Theorem~\ref{th:DoucMoulines} is fulfilled.

To verify (ii), define, for some $\epsilon > 0$, $A_N \define \sum_{i = 1}^N \ee[ |\U{i}| \condexp |\U{i}| \geq \epsilon | \partalg_N ]$. Since, as the ancestor sample is assumed to be consistent, $\max_{1 \leq i \leq N} \wgt{}{i} / \wgtsum{}$ vanishes in probability as $N$ tends to infinity, the same holds for the product $A_N \mathbbm{1} \{C \max_{1 \leq i \leq N} \wgt{}{i} > \epsilon \wgtsum{} \}$, where $C > 0$ is an arbitrary constant. On the other hand,
\begin{multline*}
    A_N \mathbbm{1} \left\{ C \max_{1 \leq i \leq N} \wgt{}{i}
    \leq \epsilon \wgtsum{} \right\} \\
    \leq \sum_{i = 1}^N \ee \left[ \left. |\U{i}| \condexp |f(\partinew{}{i})| \sum_{\ell = 1}^\alpha
    \qdrawnew{\parti{}{i}, \partinew{}{i}}{\ell}{} \geq \alpha C \right|  \partalg_{N} \right] \\
    = \wgtsum{}^{-1} \sum_{i = 1}^N \wgt{}{i} \int
    |f(\tilde{\xi})| \int_{ |f(\tilde{\xi})| \sum_{\ell = 1}^\alpha v_\ell \geq \alpha C } v_1
    \lemkern^{\varotimes \alpha}\myleft(\parti{}{i}, \tilde{\xi}, \dd v_{1:\alpha} \myright) \, \proposal{} \myleft( \parti{}{i}, \dd \tilde{\xi} \myright) \eqsp.
\end{multline*}
Now, since, for all $\xi \in \partspace$,
$$
    \int |f(\tilde{\xi})| \int_{ |f(\tilde{\xi})| \sum_{\ell = 1}^\alpha v_\ell \geq \alpha C } v_1
    \lemkern^{\varotimes \alpha}(\xi, \tilde{\xi}, \dd v_{1:\alpha}) \, \proposal{}(\xi, \dd \tilde{\xi}) \leq \ukvar(\xi, |f|) \eqsp,
$$
where $\ukvar(\cdot, |f|) \in \cset$, we conclude, using Property~(ii) of Definition~\ref{def:proper:set}, that the mapping
$$
\xi \mapsto \int |f(\tilde{\xi})| \int_{ |f(\tilde{\xi})| \sum_{\ell = 1}^\alpha v_\ell \geq \alpha C } v_1
    \lemkern^{\varotimes \alpha}(\xi, \tilde{\xi}, \dd v_{1:\alpha}) \, \proposal{}(\xi, \dd \tilde{\xi})
$$
on $\partspace$ belongs to $\cset$ as well. Thus, consistency of the ancestor sample implies that
\begin{multline} \label{eq:first:condition:limit}
    \sum_{i = 1}^N \ee \left[ \left. |\U{i}| \condexp |f(\partinew{}{i})| \sum_{\ell = 1}^\alpha \qdrawnew{\parti{}{i}, \partinew{}{i}}{\ell}{} \geq \alpha C \right| \partalg_N \right] \\
    \ConvP \iint |f(\tilde{\xi})| \int_{ |f(\tilde{\xi})| \sum_{\ell = 1}^\alpha v_\ell \geq \alpha C } v_1
    \lemkern^{\varotimes \alpha}(\xi, \tilde{\xi}, \dd v_{1:\alpha}) \, \proposal{}(\xi, \dd \tilde{\xi}) \, \nu(\xi) \eqsp.
\end{multline}
In addition, since the constant $C$ may be chosen arbitrarily large, the limit \eqref{eq:first:condition:limit} can be made arbitrarily small by the dominated convergence theorem. We hence conclude that $A_N$ tends to zero in probability as $N$ tends to infinity. This establishes \eqref{eq:key:limit-1}.

In order to establish \eqref{eq:weightconvergence} it is, by Slutsky's theorem and \eqref{eq:weight:limit}, enough to prove that
\begin{equation} \label{eq:max:key:limit}
    \wgtsum{}^{-1} \max_{1 \leq i \leq N} \wgtnew{}{i} \ConvP 0 \eqsp.
\end{equation}
Thus, take again a constant $C > 0$ and write
\begin{multline} \label{eq:cond:max:limit}
    \wgtsum{}^{-1} \max_{1 \leq i \leq N} \wgtnew{}{i} \mathbbm{1} \left\{
    \sum_{\ell = 1}^{\alpha} V^\ell \myleft(\parti{}{i}, \partinew{}{i}\myright) \geq \alpha C \right\} \\
    \leq \wgtsum{}^{-1} \sum_{i = 1}^N \wgtnew{}{i} \mathbbm{1} \left\{
    \sum_{\ell = 1}^{\alpha} \qdrawnew{\parti{}{i}, \partinew{}{i}}{\ell}{} \geq \alpha C \right\} \eqsp.
\end{multline}
To prove that the right hand side of \eqref{eq:cond:max:limit} converges, we introduce the triangular array $\Up{i} \define \wgtnew{}{i} \mathbbm{1}\{ \sum_{\ell = 1}^{\alpha} \qdrawnew{\parti{}{i}, \partinew{}{i}}{\ell}{} \geq \alpha C \} / \wgtsum{}$, $N \geq 1$, $1 \leq i \leq N$, and let the sub-$\sigma$-fields $\partalg_N$, $N \geq 1$, be defined as above. Next, we use again Theorem \ref{th:DoucMoulines}. To verify the first condition, take conditional expectation with respect to $\partalgp_N$ and reuse \eqref{eq:first:condition:limit} with $f$ being the unity function; this yields
\begin{equation*}
    \sum_{i = 1}^N \ee \left[  \left. \Up{i} \right| \partalgp_N \right]
    \ConvP \iiint_{\, \sum_{\ell = 1}^\alpha v_\ell \geq \alpha C } v_1 \lemkern^{\varotimes \alpha}(\xi, \tilde{\xi}, \dd v_{1:\alpha}) \, \proposal{}(\xi, \dd \tilde{\xi}) \, \nu(\dd \xi) \eqsp,
\end{equation*}
implying (i). To verify (ii), take an $\epsilon > 0$ and define $A_N \define \sum_{i = 1}^N \ee[ |\Up{i}| \condexp |\Up{i}| \geq \epsilon | \partalg_N ]$. Then
\begin{equation*}
    A_N = \wgtsum{}^{-1} \sum_{i = 1}^N \wgt{}{i} \ee \left[ \left. V^1(\parti{}{i}, \partinew{}{i}) \condexp \wgtnew{}{i} \geq \epsilon \wgtsum{}, \,  \sum_{\ell = 1}^{\alpha} V^\ell \myleft(\parti{}{i}, \partinew{}{i}\myright) \geq \alpha C \right| \partalgp_N \right] \eqsp,
\end{equation*}
implying that, for an arbitrary constant $C' > 0$, following the lines of \eqref{eq:first:condition:limit},
\begin{multline} \label{eq:last:C:limit}
    A_N \mathbbm{1} \left\{ C' \max_{1 \leq i \leq N} \wgt{}{i} \leq \epsilon \wgtsum{} \right\} \\
    \leq \wgtsum{}^{-1} \sum_{i = 1}^N \wgt{}{i} \ee \left[ \left. V^1(\parti{}{i}, \partinew{}{i}) \condexp \sum_{\ell = 1}^{\alpha} V^\ell \myleft(\parti{}{i}, \partinew{}{i}\myright) \geq \alpha (C \vee C') \right| \partalgp_N \right] \\
    \ConvP \iiint_{\, \sum_{\ell = 1}^\alpha v_\ell \geq \alpha (C \vee C') } v_1 \lemkern^{\varotimes \alpha}(\xi, \tilde{\xi}, \dd v_{1:\alpha}) \, \proposal{}(\xi, \dd \tilde{\xi}) \, \nu(\dd \xi) \eqsp.
\end{multline}
On the other hand,
\begin{equation*}
    \wgtsum{}^{-1} \max_{1 \leq i \leq N} \wgtnew{}{i}
    \mathbbm{1} \left\{ \sum_{\ell = 1}^{\alpha}
    \qdrawnew{\parti{}{i}, \partinew{}{i}}{\ell}{} <  \alpha C \right\}
    \leq C \wgtsum{}^{-1} \max_{1 \leq i \leq N} \wgt{}{i} \ConvP 0 \eqsp.
\end{equation*}
Thus, since the limit \eqref{eq:last:C:limit} can be made arbitrarily small by increasing $C'$, we conclude that $A_N$ tends to zero as $N$ tends to infinity. This in turn implies that the upper bound in \eqref{eq:cond:max:limit} tends to
\begin{equation} \label{eq:very:last:limit}
    \iiint_{\, \sum_{\ell = 1}^\alpha v_\ell \geq \alpha C } v_1 \lemkern^{\varotimes \alpha}(\xi, \tilde{\xi}, \dd v_{1:\alpha}) \, \proposal{}(\xi, \dd \tilde{\xi}) \, \nu(\dd \xi) \eqsp.
\end{equation}
Finally, we complete the proof by noting that \eqref{eq:very:last:limit} can be made arbitrarily small by increasing $C$. %Finally, we observe that when $C \rightarrow \infty$, the right hand side of \eqref{eq:cond:max:limit} tends to zero. This shows \eqref{eq:max:key:limit}, which completes the proof.
\end{proof}

%%%%%%%%%%%%%%%%%%%%%%%%%%%%%%%%%%
%% Back to proof of main result %%
%%%%%%%%%%%%%%%%%%%%%%%%%%%%%%%%%%

We now use Lemma~\ref{lem:particleconsistency} to prove consistency of Monte Carlo estimates produced by the \GPEAPS. For this purpose, let $\selpart{0:k|k}{i} \define \parti{0:k|k}{\idx{k}{i}}$, $1 \leq i \leq N$, denote the selected particles obtained in Step~(2) of Algorithm~\ref{alg:SMC:smoothing}. Consequently, the sample $(\selpart{0:k|k}{i})_{i = 1}^N$ is obtained by resampling the ancestor particles
$(\parti{0:k|k}{i})_{i = 1}^N$ multinomially with respect to the normalised adjusted weights $(\wgt{k}{j} \adj{k}{j} / \sum_{\ell = 1}^N \wgt{k}{\ell} \adj{k}{\ell})_{j = 1}^N$. This operation will in the following be referred to as \emph{selection}. Using this notation and terminology it is now possible to describe one iteration of the \GPEAPS\ by the following three transformations:
\begin{multline*} \label{eq:updating:scheme}
    (\parti{0:k|k}{i}{}, \wgt{k}{i})_ {i = 1}^N \overpil{\textbf{I}: Weighting}
    (\parti{0:k|k}{i}{}, \adj{k}{i} \wgt{k}{i})_{i = 1}^N \rightarrow \\
    \overpil{\textbf{II}: Selection} (\selpart{0:k|k}{i}{},1)_{i = 1}^N \overpil{\textbf{III}: Mutation}
    (\parti{0:k + 1|k + 1}{i}{}, \wgt{k+1}{i})_{i = 1}^N \eqsp.
\end{multline*}
Here the third operation refers to the random weight mutation procedure described in Algorithm~\ref{alg:SMC:rw-mutation}.

To prove Proposition~\ref{prop:cons:GPEAPS} we proceed by induction and assume that $(\parti{0:k|k}{i}{},\wgt{k}{i})_ {i = 1}^N$ is consistent for $(\sm{k}, \Lp{1}(\Xset^{k + 1}, \sm{k}))$. Next, we show how consistency is preserved through one iteration of the algorithm by analysing separately Steps~({\bf I}--{\bf III}).
%\\ \ \\

\emph{Step {\bf I}}. Define the modulated smoothing measure
\[
\modmeas{k}(A) \define \frac{\sm{k}(\adjfunc{k} \mathbbm{1}_A)}{\sm{k}(\adjfunc{k})} \eqsp, \quad A \in \Xalg^{\varotimes(n + 1)} \eqsp;
\]
then the weighting operation in Step~{\bf I} can be viewed as a transformation according Algorithm~\ref{alg:SMC:rw-mutation} with $\partspace = \Xset^{n + 1}$, $\partspacenew = \Xset^{n + 1}$, and
    \[
    \begin{cases}
        \nu = \sm{k} \eqsp, \\
        \mu = \modmeas{k} \eqsp, \\
        \proposal{}(x_{0:k}, A) = \delta_{x_{0:k}}(A) \eqsp, \\
        \ukvar(x_{0:k}, A) =  \adjfunc{k}(x_{0:k}) \, \delta_{x_{0:k}}(A) \eqsp, \\
        \lemkern(x_{0:k}, x_{0:k}', A) =  \delta_{\adjfunc{k}(x_{0:k}')}(A) \eqsp.
    \end{cases}
    \]
Thus, by applying Lemma~\ref{lem:particleconsistency} we conclude that $(\parti{0:k|k}{i}{}, \adj{k}{i} \wgt{k}{i})_{i = 1}^N$ is consistent for $\modmeas{k}$ and the (proper) set
\begin{equation*}
    \left\{ f \in \Lp{1}(\modmeas{k}, \Xset^{n + 1}) : \adjfunc{k}|f| \in \Lp{1}(\sm{k}, \Xset^{n + 1}) \right\} = \Lp{1}(\modmeas{k}, \Xset^{n + 1}) \eqsp.
\end{equation*}

\emph{Step {\bf II}}. Applying Theorem~3 in \cite{douc:moulines:2008} gives immediately that $(\selpart{0:k|k}{i}{},1)_{i = 1}^N$ is consistent for $[\modmeas{k}, \Lp{1}(\modmeas{k}, \Xset^{n + 1})]$ for both the selection schedules \eqref{eq:mult:selection} and \eqref{eq:deter:plus:res:selection}.

\emph{Step {\bf III}}. Also the third step is handled using Lemma~\ref{lem:particleconsistency}. In this case, we set $\partspace = \Xset^{n + 1}$, $\partspacenew = \Xset^{n + 2}$, and
\[
    \begin{cases}
        \nu = \modmeas{k} \eqsp, \\
        \mu = \sm{k + 1} \eqsp, \\
        \proposal{}(x_{0:k}, A) = \int_A \delta_{x_{0:k}}(\ud x_{0:k}') \,
        \proposal{k}(x_k', \ud x_{k + 1}') \eqsp, \\
        \ukvar(x_{0:k}, A) = \int_A \wgtfunc{k}(x_{0:k + 1}') \, \delta_{x_{0:k}}(\ud x_{0:k}') \,
        \proposal{k}(x_k', \ud x_{k + 1}') \eqsp, \\
        \lemkern(x_{0:k}, x_{0:k + 1}', A) \\ \qquad = \int \mathbbm{1}_A \{ v g(x_{k + 1}', Y_{k + 1}) / [\adjfunc{k}(x_{0:k}') \propdens{k}(x_k', x_{k + 1}')] \} \qkern(x_k', x_{k + 1}', \dd v) \eqsp,
    \end{cases}
\]
where $\qkern$ is the GPE described in Section~\ref{section:GPEs} (and in more detail in Appendix~\ref{section:appendix:EA}). Thus, using Lemma~\ref{lem:particleconsistency} yields that $(\parti{0:k + 1|k + 1}{i}{},\wgt{k + 1}{i})_{i = 1}^N$ is consistent for $\sm{k + 1}$ and the set
\begin{equation*}
    \left\{ f \in \Lp{1}(\sm{k + 1}, \Xset^{k + 2}) : \ukvar{}(\cdot, |f|) \in \Lp{1}(\modmeas{k}, \Xset^{n + 1}) \right\} %\iint |f(x_{0:k + 1})| \wgtfunc{k}(x_{k:k + 1}) \, \modmeas{k}(\ud x_{0:k}) \, \proposal{k}(x_k, \ud x_{k + 1}) < \infty \right\} \\
    = \Lp{1}(\sm{k + 1}, \Xset^{k + 2}) \eqsp.
\end{equation*}
Finally, we complete the proof by noting that the induction hypothesis is fulfilled for $k = 0$ by assumption.

\subsection{Proof of Theorem~\ref{th:consistency}}
\label{section:proof:Q:conv}

Decompose the error according to
\begin{multline} \label{Eq:DecompError}
    \IMQ^N_n(\theta, \theta') - \IMQ_n(\theta, \theta') \\
    = \sum_{k = 0}^{n - 1} \left[ \left( \Omega_{k(\Delta_n)}^{N, \theta'} \right)^{-1}
    \sum_{i = 1}^N \wgt{k(\Delta_n)}{i, \theta'} s_k^{\bar{\alpha}}
    \left( \parti{k:k + 1|k(\Delta_n)}{i, \theta'}, \theta \right) \right. \\
    \Bigg. - \int s_k(x_{k:k + 1} ; \theta) \, \sm{k(\Delta_n)} \left( \ud x_{k:k + 1} ; \theta' \right) \Bigg] \\
    + b_n(\Delta_n, \theta, \theta') \eqsp,
\end{multline}
where the bracket terms are errors originating from the \GPEAPS\ and the second term $b_n$, defined in \eqref{eq:def:bias}, is the cost of introducing the fixed lag. By combining Proposition~\ref{prop:cons:GPEAPS} with Slutsky's theorem we conclude that
\begin{multline} \label{eq:first:term:conv}
    \sum_{k = 0}^n \left( \Omega_{k(\Delta_n)}^{N, \theta'} \right)^{- 1}
    \sum_{i = 1}^N \wgt{k(\Delta_n)}{i, \theta'}
    \log g_\theta \left( \parti{k | k(\Delta_n)}{i, \theta'}, Y_k \right)
    \\ \ConvP \sum_{k = 0}^n \int \log g_\theta \left(x_k, Y_k \right) \, \sm{k(\Delta_n)}(\dd x_k ; \theta') \eqsp,
\end{multline}
as $x_{0:k(\Delta_n)} \mapsto \log g_\theta(x_k, Y_k)$ belongs to $\Lp{1}(\sm{k(\Delta_n)}(\cdot; \theta'), \Xset^{k(\Delta_n) + 1})$ by assumption. Thus, the second term of the intermediate quantity estimator \eqref{eq:def:Q:est} is consistent. In order to establish consistency of the complete estimator it remains to prove that
\begin{multline} \label{eq:second:term:conv}
\sum_{k = 0}^{n - 1} \left( \bar{\alpha} \Omega_{k(\Delta_n)}^{N, \theta'} \right)^{- 1}
    \sum_{i = 1}^N \wgt{k(\Delta_n)}{i, \theta'} \sum_{\ell = 1}^{\bar{\alpha}} \bar{V}_\theta^\ell \left( \parti{k:k + 1|k(\Delta_n)}{i, \theta'} \right) %\left/ \bar{\alpha} \Omega_{k(\Delta_n)}^{N, \theta'} \right.
    \\ \ConvP \sum_{k = 0}^{n - 1} \int \log q_\theta\left(x_k, x_{k + 1}\right) \, \sm{k(\Delta_n)}(\dd x_{k:k + 1} ; \theta') \eqsp.
\end{multline}
To do this, we define $\Utd{i} \define \wgt{k(\Delta_n)}{i, \theta'} \sum_{\ell = 1}^{\bar{\alpha}} \bar{V}_\theta^\ell (\parti{k:k + 1|k(\Delta_n)}{i, \theta'}) / \bar{\alpha} \Omega_{k(\Delta_n)}^{N, \theta'}$ and $\partalgtd_N \define \sigma\{ (\parti{0:k(\Delta_n)|k(\Delta_n)}{i, \theta'}, \wgt{k(\Delta_n)}{i, \theta'})_{i = 1}^N \}$ and appeal to Theorem~\ref{th:DoucMoulines} and Proposition~\ref{prop:cons:GPEAPS}. Since $\log q_\theta(x_k, x_{k + 1}) \leq \int |v| \, \lqkern_\theta(x_k, x_{k + 1}, \dd v)$ for all $x_{k:k + 1} \in \Xset^2$, the mapping $x_{0:k(\Delta_n)} \mapsto \log q_\theta(x_k, x_{k + 1})$ belongs to $\Lp{1}(\sm{k(\Delta_n)}(\cdot; \theta'), \Xset^{k(\Delta_n) + 1})$. Hence,
\begin{multline} \label{eq:key:limit-2}
    \sum_{i = 1}^N \ee \left[ \Utd{i} \left| \partalgtd_N \right. \right]
    = \left( \Omega_{k(\Delta_n)}^{N, \theta'} \right)^{-1}
    \sum_{i = 1}^N \wgt{k(\Delta_n)}{i, \theta'} \log q_\theta\left( \parti{k:k + 1|k(\Delta_n)}{i, \theta'} \right) %\left/ \Omega_{k(\Delta_n)}^{N, \theta'} \right.
    \\ \ConvP \int \log q_\theta(x_k, x_{k + 1}) \, \sm{k(\Delta_n)}(\dd x_{k:k + 1} ; \theta') \eqsp,
\end{multline}
from which we conclude that \eqref{eq:second:term:conv} may be established by verifying the two assumptions of Theorem~\ref{th:DoucMoulines}. Following \eqref{eq:key:limit-2} and using again that $x_{0:k(\Delta_n)} \mapsto \int |v| \lqkern_\theta(x_k, x_{k + 1}, \dd v)$ belongs to $\Lp{1}(\sm{k(\Delta_n)}(\cdot; \theta'), \Xset^{k(\Delta_n) + 1})$ by assumption, we conclude that
\begin{equation*}
    \sum_{i = 1}^N \ee \left[ \left| \Utd{i} \right| \left| \partalgtd_N \right. \right] \ConvP \iint |v| \, \lqkern_\theta(x_k, x_{k + 1}, \dd v) \, \sm{k(\Delta_n)}(\dd x_{k:k + 1} ; \theta') \eqsp,
\end{equation*}
which verifies Assumption~(i) (by tightness of sequences converging in probability). To verify (ii), let $\epsilon > 0$ and set $\bar{A}_N \define \sum_{i = 1}^N \ee[ |\Utd{i}| \condexp |\Utd{i}| \geq \epsilon | \partalgtd_N ]$. Then, for any constant $C > 0$, by consistency of the particle sample,
\begin{equation}
    \bar{A}_N \mathbbm{1}\left\{ C \max_{1 \leq i \leq N} \wgt{k(\Delta_n)}{i, \theta'} > \epsilon \Omega_{k(\Delta_n)}^{N, \theta'} \right\} \ConvP 0 \eqsp.
\end{equation}
On the other hand,
\begin{multline*}
    \bar{A}_N \mathbbm{1}\left\{ C \max_{1 \leq i \leq N} \wgt{k(\Delta_n)}{i, \theta'} \leq \epsilon \Omega_{k(\Delta_n)}^{N, \theta'} \right\} \\
    \leq \sum_{i = 1}^N \ee \left[ \left. |\Utd{i}| \condexp \left| \sum_{\ell = 1}^{\bar{\alpha}} \bar{V}_\theta^\ell \left( \parti{k:k + 1|k(\Delta_n)}{i, \theta'} \right) \right| \geq C \bar{\alpha} \right| \partalgtd_N \right] \\
    \leq \left( \Omega_{k(\Delta_n)}^{N, \theta'} \right)^{-1} \sum_{i = 1}^N \wgt{k(\Delta_n)}{i, \theta'} \int_{\, | \sum_{\ell = 1}^{\bar{\alpha}} v_\ell | \geq C \bar{\alpha}} |v_1| \, \lqkern_\theta^{\varotimes \bar{\alpha}}(x_k, x_{k + 1}, \ud v_{1 : \bar{\alpha}}) \eqsp.
\end{multline*}
Now, since, for all $x_{k:k + 1} \in \Xset^2$,
$$
    \int_{\, | \sum_{\ell = 1}^{\bar{\alpha}} v_\ell | \geq C \bar{\alpha}} |v_1| \, \lqkern_\theta^{\varotimes \bar{\alpha}}(x_k, x_{k + 1}, \ud v_{1 : \bar{\alpha}}) \leq \int |v| \, \lqkern_\theta(x_k, x_{k + 1}, \ud v) \eqsp,
$$
we get, using Proposition~\ref{prop:cons:GPEAPS},
\begin{multline} \label{eq:final:epsilon:limit}
\left( \Omega_{k(\Delta_n)}^{N, \theta'} \right)^{-1} \sum_{i = 1}^N \wgt{k(\Delta_n)}{i, \theta'} \int_{\, | \sum_{\ell = 1}^{\bar{\alpha}} v_\ell | \geq C \bar{\alpha}} |v_1| \, \lqkern_\theta^{\varotimes \bar{\alpha}}(x_k, x_{k + 1}, \ud v_{1 : \bar{\alpha}}) \\
\ConvP \iint_{\, | \sum_{\ell = 1}^{\bar{\alpha}} v_\ell | \geq C \bar{\alpha}} |v_1| \, \lqkern_\theta^{\varotimes \bar{\alpha}}(x_k, x_{k + 1}, \ud v_{1 : \bar{\alpha}}) \, \sm{k(\Delta_n)}(\ud x_{k:k + 1}; \theta') \eqsp.
\end{multline}
We now note that the limit in \eqref{eq:final:epsilon:limit} can be made arbitrarily small by increasing $C$. This verifies condition (ii) in Theorem~\ref{th:DoucMoulines}, which completes the proof of \eqref{eq:second:term:conv}. Finally, combining \eqref{eq:second:term:conv} with \eqref{eq:first:term:conv} completes the proof of Theorem~\ref{th:consistency}.

\section{More on the GPE}
\label{section:appendix:EA}
The outline of this section follows \citet{beskos:papaspiliopoulos:roberts:fearnhead:2006} and \citet{fearnhead:papaspiliopoulos:roberts:2008}, and we limit our scope to the one-dimensional case; multivariate extensions are treated by \citet{beskos:papaspiliopoulos:roberts:2008}. Let $(C[0,t], \mathcal{C}[0, t])$ be the measurable space of continuous functions on $[0, t]$ and denote by $\ss{x}_\theta$ the law of $\Xnew{}$ on $(C[0,t], \mathcal{C}[0, t])$ for the initial condition $\Xnew{0} = W_0 = x$. Also, let $\bb{t, x, x'}$ be the law, on the same space, of the Brownian bridge process $\path = (\path_s)_{0 \leq s \leq t}$ starting in $x$ at time zero and ending in $x'$ at time $t$. Similarly, denote by $\ss{t, x, x'}_\theta$ the law of the \emph{diffusion bridge} obtained when $\Xnew{}$ is conditioned to start at $\Xnew{0} = W_0 = x$ and to finish at $\Xnew{t} = x'$. Recall the definition \eqref{eq:def:alpha} of $\alpha(\cdot, \theta)$ and let
\begin{equation*}
    A(u, \theta) \define \int^u \alpha(v, \theta) \, \ud v
\end{equation*}
be any antiderivative of $\alpha(\cdot, \theta)$. The role of Assumptions {\it (A\ref{hyp:Beskos:C0}--A\ref{hyp:Beskos:C2})} is to guarantee that $\ss{t, x, x'}_\theta$ is absolutely continuous with respect to $\bb{t, x, x'}$ with Radon-Nikodym derivative
\begin{multline} \label{Eq:RN_ds/dw:bridges}
    \frac{\dd \ss{x, x', t}_\theta}{\dd \ww{x, x', t}}(w) \\
    = \frac{\normpdf_t(x' - x)}{\tilde{q}_\theta(x, x', t)}
    \exp \left(  A(x', \theta) - A(x, \theta)
    - \frac{1}{2} \int_0^t (\alpha^2 + \alpha')(w_s, \theta) \,
    \dd s \right) \eqsp,
\end{multline}
where $w \in C[0, t]$ and $\normpdf_t$ denotes the density function of the zero mean normal distribution with variance $t$. Now, define, for $u \in \rr$, the \emph{drift functional}
\begin{equation*}
    \phi(u, \theta) \define \frac{\alpha^2(u, \theta) + \alpha'(u, \theta)}{2} - l(\theta) \eqsp,
\end{equation*}
where $l(\theta)$ is the lower bound given in Assumption~\refhyp{Beskos:C2}. The transition density $\tilde{q}_\theta$ can, using \eqref{Eq:RN_ds/dw:bridges}, be expressed as
\begin{multline*}
    \tilde{q}_\theta(x, x', t) = \normpdf_t(x' - x)\exp \left( A(x', \theta) - A(x, \theta) - l(\theta) t \right) \\ \times
    %\ee_{\bb{t,x,x'}} \left [ \exp \left( - \int_0^t  \phi(\path_s,\theta) \, \dd s \right) \right ] \eqsp,
    \int \exp \left( - \int_0^t \phi(w_s, \theta) \, \dd s \right) \, \bb{t,x,x'}( \ud w) \eqsp,
\end{multline*}
Accordingly, we wish to calculate expectations of the form
\begin{equation}\label{eq:expectation:exp:integral}
    %\ee_{\bb{t,x,x'}} \left [ \exp \left( - \int_0^t  f(\path_s) \, \dd s \right) \right ]\eqsp, \\
    \int \exp \left( - \int_0^t f(w_s) \, \dd s \right) \, \bb{t, x, x'}(\ud w) \eqsp.
\end{equation}
Now assume that it is possible to simulate simultaneously a pair $(\Lower{f}, \Upper{f})$ of random variables and a trajectory $(\tilde{W}_s)_{s = 0}^t$ such that
\[
    \Lower{f} \leq f(\tilde{W}_s) \leq \Upper{f} \eqsp, \quad \mathrm{for\ all\ } s \in [0, t] \eqsp;
\]
in practice this will most often be carried through by first simulating a maximum and a minimum of the Brownian bridge process $\tilde{W}$ and hereafter interpolating, using Bessel bridges, the rest of the bridge conditionally on these. Let $\kappa$ be a discrete random variable having, conditionally on $\tilde{W}_f^{\pm}$, probability distribution $p_t( \cdot | \tilde{W}_f^{\pm})$. Then it is easily established that the GPE
\begin{equation*}
    \exp(- \Upper{f} t) \frac{t^\kappa}{\kappa! p_t(\kappa | \tilde{W}_f^{\pm})} \prod_{\ell = 1}^\kappa [ \Upper{f} - f(\path_{\psi_\ell}) ]
\end{equation*}
(associated with $p_t$) is an unbiased estimator of \eqref{eq:expectation:exp:integral}. Here $(\psi_\ell)_{\ell \geq 1}$ are mutually independent variables that are uniformly distributed over $[0, t]$ and independent of $\mathcal{F}_t$. Note that the distribution $p_t$ can be chosen freely, yielding a \emph{whole class} of GPEs, and an optimal choice is discussed by \citet{fearnhead:papaspiliopoulos:roberts:2008}. In all applications considered in this paper we will use let $\kappa$ be Poisson-distributed. % with expectation $(\Upper{f} - \Lower{f})t$.

Using the Girsanov theorem, it can be shown that
\begin{multline} \label{eq:loglik:contobs}
    \log \tilde{q}_t(x, x') =  - \frac{1}{2} \log(2 \pi t)  -\frac{(x' - x)^2}{2t}\\ + A(x', \theta) - A(x, \theta) - l(\theta) t - \int \left( \int_0^t  \phi(w_s,\theta) \, \dd s \right) \ss{x,x',t}(\dd w) \eqsp,
\end{multline}
Since the right hand side of \eqref{Eq:RN_ds/dw:bridges} can be bounded from above and below, a rejection sampler producing samples from the diffusion bridge can be constructed. This is possible as the right hand side of \eqref{Eq:RN_ds/dw:bridges} is proportional to the probability that a marked Poisson process on $[0, t] \times [0, 1]$ with intensity $r \define \sup_x \{ \phi(x) ; \Lower{\phi} < x < \Upper{\phi} \}$ is below the graph $s \mapsto \phi(\path_s; \theta)/r$. However, while observing the path for all $s$ is impossible, a finite construction can be devised by sampling the Brownian bridge at points specified by the marked Poisson process; we refer to \citet{beskos:papaspiliopoulos:roberts:fearnhead:2006} for details. The algorithm is described by the following.

\bigskip

\begin{algorithm}{XXX}{\label{alg:diff:bridge:sampling}
    \qcomment{Sampling a skeleton of a diffusion bridge}}
    simulate an outcome $(\chi_\ell,\psi_\ell)_{\ell = 1}^{\kappa}$ of the marked Poisson process with intensity $r$ and $\kappa \sim \mathrm{Po}(r)$; \\
    conditional on $\tilde{W}_{\phi}^{\pm}$, simulate $(\path_{\chi_{\ell}})_{\ell = 1}^{\kappa}$; \\
    \qif $\phi(\path_{\chi_{\ell}})/r < \psi_\ell$ \\
    \qthen \qreturn $(\path_{\chi_{\ell}})_{\ell = 1}^{\kappa}$ \\
    \qelse go to (1) \qfi
\end{algorithm}

\bigskip

By interpolating the returned skeleton $(\path_{\chi_{\ell}})_{\ell = 1}^{\kappa}$, samples $\path_{u}$, with $(\path_s)_{s = 0}^t \sim \ss{x,x',t}$, can be obtained for any $0 \leq u \leq t$. Given samples from the diffusion bridge, an unbiased estimator of \eqref{eq:loglik:contobs} can be straightforwardly constructed in the following way. Let $\psi \sim \Unif(0,t)$ be independent of $\mathcal{F}_t$. Then $- t \phi(\path_\psi,\theta)$ is an unbiased estimator of $\int ( \int_0^t  \phi(w_s,\theta) \, \dd s ) \, \ss{x,x',t}(\dd w)$ since
\begin{multline*}
    \ee \left[ t \phi(\path_\psi,\theta) \right] = \ee \left[ \ee \left[ \left. t \phi(\path_\psi,\theta) \right| \mathcal{F}_t \right] \right] \\ = \ee \int_0^t \phi(\path_s,\theta) \, \dd s  = \int \left( \int_0^t  \phi(w_s,\theta) \, \dd s \right) \, \ss{x,x',t}(\dd w) \eqsp.
\end{multline*}
Finally, plugging this estimator into \eqref{eq:loglik:contobs} yields an unbiased estimator of $\log \tilde{q}_t$.  

\bibliographystyle{apalike}
\bibliography{Reference}
%\bibliography{/Users/jonasstrojby/Documents/TeX/Jonas_ref}

\end{document}